\documentclass[preprint]{elsarticle}

\usepackage{times}
\usepackage{lineno}

\journal{Journal of \LaTeX\ Templates}

\usepackage{amstext,amscd,amsmath}
\usepackage{amsthm}
\usepackage{amsfonts}
\usepackage{ifpdf}
\usepackage{graphicx,amssymb}
\usepackage{latexsym,bm}
\usepackage{mathrsfs}
\usepackage{enumerate}
\usepackage{color}

\newtheorem{theorem}{Theorem}[section]
\newtheorem{lemma}[theorem]{Lemma}

\theoremstyle{definition}

\newtheorem{proposition}[theorem]{Proposition}

\theoremstyle{remark}

\newtheorem{corollary}[theorem]{Corollary}

\numberwithin{equation}{section}

\begin{document}

\begin{frontmatter}

\title{The mixed spectral problem for radial Schr\"{o}dinger operators and Paley-Wiener spaces}

\author{Yongjiang Duan$^{1,\dag}$~~~  Yufei Li $^{2,\S}$ ~~~Zeguang Liu$^{2,\ddag}$}

\begin{abstract}
The inverse spectral problem for the radial Schr\"{o}dinger operators on the finite interval is investigated.
The potential is recovered from the given eigenvalues plus its information on a smaller interval.
The method is to discuss the connections between the Paley-Wiener spaces and the potentials, from which we completely determine the potential on the whole interval from the potential on a smaller interval and a set of eigenvalues in terms of the complete exponential systems. 
\end{abstract}


\end{frontmatter}

\renewcommand{\thefootnote}{\fnsymbol{footnote}}

\noindent\footnote[0]{\small {$^1$ Department of Mathematics, Jinan University, Guangzhou,  510632, China
}}
\noindent\footnote[0]{\small{$^{2}$ School of Mathematics and Statistics, Northeast Normal University, Changchun, 130024, China}} \\

\footnote[0] {$\dag$ yjduan@jnu.edu.cn}

\footnote[0]{$\S $ liyf495@nenu.edu.cn}

\footnote[0]{$\ddag$ Corresponding author: Zeguang Liu, liuzeguang205@nenu.edu.cn}

\vspace{-0.15in}

\section{Introduction}

Consider the radial Schr\"{o}dinger operators (or perturbed Bessel operators) (see, e.g., \cite{Tes2009, Wei1987})
\begin{align*}
L(q)(u)
:=-u^{\prime \prime}(x)+\frac{\ell(\ell+1)}{x^{2}}u(x)+q(x)u(x),
\end{align*}
defined on the interval $(0,1]$ and subject to the boundary conditions
\begin{align}\label{e1.2}
\lim_{x \rightarrow 0}\frac{u(x)}{x^{\ell+1}}<\infty, 
\qquad
u^{\prime}(1)+\beta u(1)=0,
\end{align}
where $\ell>-1/2$, the potential $q$ is a real-valued function in $L^{1}(0,1)$, and the parameter $\beta\in\mathbb{R}\cup\{\infty\}$. If $\beta=\infty$, then the boundary condition $u^{\prime}(1)+\beta u(1)=0$ turns out to be the Dirichlet condition $u(1)=0$. Note that when $\ell=0$, $L(q)$ reduces to a Sturm-Liouville operator
with the Dirichlet boundary condition at $x = 0$.

The operator $L(q)$ is self-adjoint on $L^2(0,1)$, see \cite[Theorem 2.4]{KST2010}. The associated spectral problem is defined as follows: 
find $\lambda \in \mathbb{C}$ and a nonzero function $u$ satisfying the above boundary conditions \eqref{e1.2} such that
\begin{align}\label{e1.1}
-u^{\prime \prime}(x)+\frac{\ell(\ell+1)}{x^{2}}u(x)+q(x)u(x)= \lambda u.
\end{align}
It is well-known (see \cite[(1.5)]{XYB2023}) that equation \eqref{e1.1} has a unique solution $\phi_{\ell}(\lambda,x,q)$ such that
\begin{align*}
\lim_{x\rightarrow 0}\frac{\phi_{\ell}(\lambda,x,q)}{x^{\ell+1}}
=\frac{\sqrt{\pi}}{2^{\ell+1}\Gamma(\ell+\frac{3}{2})}.
\end{align*}
The operator $L(q)$ corresponding to \eqref{e1.2} is bounded below and has a simple discrete spectrum $\sigma(\ell, q, \beta)$. The eigenvalue $\lambda_{\ell, \beta, n}(q)\in\sigma(\ell, q, \beta)$ (the list was sorted in ascending order) has the following asymptotic behavior:
      \begin{align}
      \sqrt{\lambda_{\ell, \infty, n}(q)}&=\left(n+\frac{\ell}{2}\right) \pi+O\left(\frac{1}{n}\right), ~~n\rightarrow\infty,\label{ew1}\\
      \sqrt{\lambda_{\ell, \beta, n}(q)}&=\left(n+\frac{\ell-1}{2}\right) \pi+O\left(\frac{1}{n}\right), ~~\beta \in \mathbb{R},
      ~~n\rightarrow\infty,\label{ew2}
      \end{align}
      and the corresponding eigenfunction is $\phi_{\ell}(\lambda_{\ell, \beta, n}(q), x, q)$, see \cite[Theorems 2.4 and 2.5]{KST2010}.
Moreover, given $q\in L^{1}(0, 1)$, the equation \eqref{e1.1} has a solution $\phi_{\ell}(\lambda, x, q)$ which is an entire function of $\lambda\in\mathbb{C}$ and  has the asymptotic expansions as $\lambda \rightarrow \infty$,
       \begin{align}
       \phi_{\ell}(\lambda, x, q)
       &=\lambda^{-\frac{\ell+1}{2}}
       \left(\sin \left(\sqrt{\lambda} x-\frac{\ell \pi}{2}\right)+O\left(|\lambda|^{-\frac{1}{2}} e^{x| \textnormal{Im}(\sqrt{\lambda}) \mid}\right)\right),\label{e2.1}\\
       \phi_{\ell}^{\prime}(\lambda, x, q)
       &=\lambda^{-\frac{\ell}{2}}
       \left(\cos \left(\sqrt{\lambda} x-\frac{\ell \pi}{2}\right)+O\left(|\lambda|^{-\frac{1}{2}} e^{x|\textnormal{Im}(\sqrt{\lambda})|}\right)\right),\label{e2.2}
       \end{align}
       for every fixed $x\in(0,1]$, see \cite[Lemma 2.2, (2.24) and (2.25)]{KST2010}. 

Inverse problems for the radial Schr\"{o}dinger operators $L(q)$ on a finite interval have been extensively studied.
It is well known that the potential is uniquely determined by
spectral measure (or the eigenvalues and the corresponding norming constants,
or the Weyl-Titchmarsh $m$-function).
Moreover, the two sets of eigenvalues corresponding to the radial Schr\"{o}dinger operators under different boundary conditions can also uniquely determine
the potential function.
These inverse problems are studied, for example, in \cite{Borg1946,Guli2001,Hat2021,HB1978,Levit1987,Marc1986,PT1987} for the case $\ell = 0$, and in \cite{SRYal2007, Carl1997,Eckh2014,Guli2005,Gasy1965,HS2010,KST2010,Ser2007,LAVS1994} for more general values of $\ell$. Inverse problem theory also includes the use of other types of spectral data to solve the inverse spectral problem for the operator $L(q)$, for more on the theory see \cite{Aceto2008, CS1994, CB2025, Chr1993,XH2025,XY2024}.

The problem considered in this paper originates from a seminal work by Hochstadt and Lieberman on the Sturm-Liouville operator. They introduced the concept of using mixed data to recover the potential, a problem now known as the mixed spectral problem (see \cite{HB1978}). This problem was later extended to the radial Schrödinger operator by Xu, Yang and Bondarenko in \cite{XYB2023}.

\textbf{The mixed spectral problem.}
\emph{Suppose that the potential $q\in L^{p}(0,1)$ is known a priori on the interval $(a, 1)$ for some $a \in(0,1]$. Recover the potential $q(x)$ on $(0, a)$ from the eigenvalues $\{\lambda_n\}_{n=1}^{\infty}$ belonging to the various spectra: $\lambda_{n}\in\sigma(\ell, q, \beta_{n})$ with different values of $\beta_{n}\in\mathbb{R}\cup\{\infty\}$ in the boundary condition \eqref{e1.2}.}

Our motivation for considering this problem mainly comes from the following two points. Firstly, for $\ell=0$,  a series of remarkable studies have been carried out. Borg \cite{Borg1946} proposed that two spectra uniquely determines the potential. Hochstadt and Lieberman \cite{HB1978} showed that the potential is uniquely determined by knowledge of the potential on half of the interval together with a spectra. Gesztesy and Simon \cite{FB2000} extended this result by proving that the potential on a longer subinterval, together with a spectral subset for which the number of eigenvalues is proportional to $a$, also uniquely determines the potential. 
del Rio, Gesztesy and Simon \cite{dGS1997} showed that if the spectra associated with three distinct boundary conditions are known, then any subset of these spectra containing at least two-thirds of the total number of eigenvalues from the three spectra is sufficient to uniquely determine the potential. Horv\'{a}th \cite[Theorem 1.1]{Horv2005} provided a complete characterization of the mixed spectral problem. Recently, Makarov and Poltoratski \cite[Theorem 6]{MP2019} gave a short proof of Horv\'{a}th's result based on the connection between mixed spectral problems for Schrödinger operators and the Beurling–Malliavin problem on completeness of exponential systems. In particular, a simple proof for the ``un-mixed" case $a=1$ was given by Baranov, Belov and Poltoratski \cite[Theorem 2]{BBP2017}, where the main idea to establish the connection between potential and Paley-Wiener space was inspired by the work of Korotyaev \cite{Koro2004}.

Secondly,  for $\ell\in\mathbb{N}$ and a potential $q\in L^{p}(0,1)$ with $p\in(1,\infty)$, Xu, Yang and Bondarenko 
provided a sufficient condition for the mixed spectral problem, and under some additional hypothetical condition, their sufficient condition is also necessary (see \cite[Theorem 1.2 and Theorem 5.2]{XYB2023}). The main tool of their proof is the singular transformation operator representation for the solution of the radial Schr\"{o}dinger equation. 

So a natural question is whether the additional hypothetical condition of \cite{XYB2023} can be removed. For $p=2$, inspired by the idea of \cite{BBP2017} and \cite{MP2019}, we give an answer of the mixed spectral problem without any additional hypothetical condition. To state our results, we need more notations.

Recall \cite[pp.9]{KST2010} that the Weyl function for $L(q)$ is defined by
\begin{align}\label{e1.3}
m(\lambda)=\frac{\phi_{\ell}(\lambda, 1, q)}
{\phi_{\ell}^{\prime}(\lambda, 1, q)},
\end{align}
which is analytic in the upper half-plane and maps the upper half-plane to itself, i.e., it is a Herglotz function. It is well known that $m$ uniquely determines the potential $q$ and the constant $\ell$, see \cite[pp.9]{KST2010}.

We can now state the main result of the paper.

\begin{theorem}\label{t1.1} 
Suppose $\Lambda$ is a sequence of distinct positive real numbers, $\ell\in\mathbb{N}_{0}$ and $a \in(0,1]$. Then the following statements are equivalent.

\begin{enumerate}

\item [\textnormal{(1)}] If $q$, $\tilde{q} \in L^{2}(0,1)$ satisfy $q=\tilde{q}$ on $(a,1)$ and $m=\tilde{m}$ on $\Lambda$, then $q=\tilde{q}$ a.e. on $(0,1)$.
    
\item [\textnormal{(2)}] The system 
$\{e^{i\lambda t}\mid\lambda\in\pm\sqrt{\Lambda}\cup\{a_n\}_{n=1}^{2\ell+2}$ is complete in $L^{2}(-2a,2a)$,
    
\end{enumerate}
where $\tilde{m}$ is the Weyl function for $L(\tilde{q})$, 
$\pm\sqrt{\Lambda}=\{z: z^{2} \in \Lambda\}$,
$a_n\neq a_k$ for $n\neq k$, and $\pm\sqrt{\Lambda}\cap\{a_n\}_{n=1}^{2\ell+2}=\emptyset$.

In particular, if $a = 1$, then \textnormal{(1)} reduces to the following: If $q$, $\tilde{q} \in L^{2}(0,1)$ with $m=\tilde{m}$ on $\Lambda$, then $q=\tilde{q}$ a.e. on $(0,1)$.
\end{theorem}

Theorem \ref{t1.1} can be restated as follows.

\begin{theorem}\label{c1.1} 
Suppose  $\ell\in\mathbb{N}_{0}$ and $a\in(0,1]$. Suppose $\lambda_{n}\in\sigma(\ell,q,\beta_{n})
\cap\sigma(\ell,\tilde{q}$, $\beta_{n}),\beta_{n}\in\mathbb{R} \cup\{\infty\}$ and $\lambda_{n}>0$ for $n\in \mathbb{N}$. Write $\Lambda=\{\lambda_{n}\}_{n=1}^{\infty}$.
Then the following statements are equivalent.
\begin{enumerate}

\item [\textnormal{(1)}] If $q$, $\tilde{q} \in L^{2}(0,1)$ satisfy $q=\tilde{q}$ on $(a,1)$, then $q=\tilde{q}$ a.e. on $(0,1)$.
    
\item [\textnormal{(2)}]The system 
$\{e^{i\lambda t}\mid\lambda\in\pm\sqrt{\Lambda}\cup\{a_n\}_{n=1}^{2\ell+2}$ is complete in $L^{2}(-2a,2a)$.
    
\item [\textnormal{(3)}] The system 
$\{t^{2 k}\mid k=0,1,2,\cdots,\ell\}\cup\{\cos(2\sqrt{\lambda} t)\mid\lambda\in \Lambda\}$ is complete in $L^{2}(0,a)$.
\end{enumerate}

\end{theorem}

When $\ell=0$, Theorems \ref{t1.1} and \ref{c1.1} are two versions of the well-known Horv\'{a}th's theorem. Theorem \ref{t1.1} is due to Makarov and Poltoratski \cite[Theorem 6]{MP2019}; the special case $a=1$ appears early in the work of Baranov, Belov and Poltoratski \cite[Theorem 2]{BBP2017}. Theorem \ref{c1.1} was proved by Horv\'{a}th \cite[Theorem 1.1]{Horv2005}. For any non-negative integer $\ell$, Theorem \ref{c1.1} shows that the sufficient condition given by Xu, Yang and Bondarenko in \cite[Theorem 1.2]{XYB2023} is also necessary. In other words, the assumptions in \cite[Theorem 5.2]{XYB2023} are removable.

For $s\in\mathbb{R}$, we denote by $\lfloor s\rfloor$ the nearest integer to $s$, with the tie-breaker favoring the smaller integer.
For general $\ell$, we provides a sufficient condition for the uniqueness of the solution of the mixed spectral problem as follows. 

\begin{theorem}\label{t1.2}
Let $\ell>-1/2$, $a \in(0,1]$.
Suppose $\Lambda$ is a sequence of distinct positive real numbers satisfying that
$\{e^{i\lambda t}\mid\lambda\in\pm\sqrt{\Lambda}
\cup\{a_n\}_{n=1}^{\lfloor2\ell\rfloor+1}$
is complete in $L^{2}(-2a,2a)$, 
where 
$\pm\sqrt{\Lambda}=\{z: z^{2} \in \Lambda\}$,
$a_n\neq a_k$ for $n\neq k$, and $\pm\sqrt{\Lambda}\cap\{a_n\}_{n=1}^{\lfloor 2\ell\rfloor+1}=\emptyset$.
If $q$, $\tilde{q} \in L^{p}(0,1)$ with $p\in [1,\infty]$ satisfy  $q=\tilde{q}$ on $(a,1)$ and $m=\tilde{m}$ on $\Lambda$, then $q=\tilde{q}$ a.e. on $(0,1)$.
\end{theorem}

As corollaries to Theorem \ref{t1.2}, we obtain generalizations of the classical results of Borg \cite{Borg1946} and Hochstadt-Lieberman \cite{HB1978} to the perturbed Bessel equation \eqref{e1.1}. Specifically, we establish the uniqueness of recovering the potential $q(x)$ on $(0,1)$ from the two spectra $\sigma(\ell,q,0)$ and $\sigma(\ell,q,\infty)$. We also study the half-inverse problem.
According to Koyunbakan and Panakhov \cite{HE2005}, if $q(x)$ is known on the half-interval $(1/2,1)$, it can be recovered on $(0,1/2)$ from a single spectrum $\sigma(\ell, q, \beta)$.
Corollary \ref{exa2} strengthens this result by demonstrating that not all data in $\sigma(\ell, q, \beta)$ are necessary.

The paper is organized as follows. 
In Section \ref{s2} we discuss in detail the properties of the spectra of the radial Schr\"{o}dinger operators, and some properties of the Paley-Wiener spaces that will be used subsequently.
In Section \ref{s8} we describe the connections between the Paley-Wiener spaces and the potentials.
In Section \ref{s4} the proofs of Theorems \ref{t1.1}-\ref{c1.1} are given. 
In Section \ref{s5} we condsider the non-integer case of $\ell$ and present some applications of Theorem \ref{t1.2}, namely, the Borg-type and Hochstadt-Lieberman-type results.
In \ref{s6}, the product formulas of the solutions are given.

\section{Preliminaries}\label{s2}

\subsection{The radial Schr\"{o}dinger operators}

The aim of this subsection is to discuss the properties of the spectra of the radial Schr\"{o}dinger operators.

The following lemma shows that $\{\lambda_{\ell, \beta, n}(q)\}$ and $\{\lambda_{\ell, \infty, n}(q)\}$ are interlacing.

\begin{lemma}\label{p2.3}
Suppose $q\in L^{1}(0,1)$, $\ell\in\mathbb{N}_{0}$ and $\beta\in\mathbb{R}$. Then 
\begin{align*}
\lambda_{\ell, \beta, n}(q)<\lambda_{\ell, \infty, n}(q)<\lambda_{\ell, \beta, n+1}(q),
~~n \geq 1.
\end{align*}
\end{lemma}

\begin{proof}
It follows from \cite[Lemma 9.2]{Tes2009} that there
exist two linearly independent nonzero solutions $y_{1}(\lambda,x, q)$ and $y_{2}(\lambda,x, q)$ of the equation \eqref{e1.1} such that 
\begin{align*}
\phi_{\ell}(\lambda,x, q)
=\phi_{\ell}(\lambda, 1, q)y_{1}(\lambda,x, q)
+\phi_{\ell}^{\prime}(\lambda, 1, q)y_{2}(\lambda,x, q),
\end{align*}
which implies $\phi_{\ell}(\lambda, 1, q)$ and $\phi_{\ell}^{\prime}(\lambda, 1, q)+\beta\phi_{\ell}(\lambda, 1, q)$ have no common zeros. 

Write
\begin{align*}
m_{\beta}(\lambda)=-\frac{1}{m(\lambda)}-\beta
=-\frac{\phi_{\ell}^{\prime}(\lambda,1,q)+\beta\phi_{\ell}(\lambda,1,q)}
{\phi_{\ell}(\lambda, 1, q)}.
 \end{align*}
Recall that $m$ is a Herglotz function, and so is $m_{\beta}$. Also, combining \eqref{e2.5} and \eqref{e2.6}, we get $m_{\beta}(\lambda)=\overline{{m}_{\beta}(\bar{\lambda})}$ and $\{\lambda_{\ell, \beta, n}(q)\}$ and $\{\lambda_{\ell, \infty, n}(q)\}$ are the zeros and poles of $m_{\beta}$, respectively. Thus from \cite[Theorem 1, pp.308]{Lev1964}, it follows that $\{\lambda_{\ell, \beta, n}(q)\}$ and $\{\lambda_{\ell, \infty, n}(q)\}$ are interlacing. Moreover, by \eqref{ew1} and \eqref{ew2}, we get 
\[
\lambda_{\ell, \beta, n}(q)<\lambda_{\ell, \infty, n}(q)<\lambda_{\ell, \beta, n+1}(q),
~~n \geq 1.
\]
This completes the proof.
\end{proof}

In particular, for $\ell\in\mathbb{N}_{0}$ and $q\in L^{2}(0,1)$,  \cite[Proposition 3.1]{Ser2007} gives   
\begin{align}\label{e2.9}
\lambda_{\ell,\infty,n}(q)
&=\left(n+\frac{\ell}{2}\right)^{2}\pi^{2}+\int_{0}^{1}q(t)dt 
-\ell(\ell+1)+l^{2}(n),
\end{align}
 and \cite[Theorems 1.2]{Carl1993} combined with \cite[Theorems 1.1]{Carl1997} yields
\begin{align*}
\lambda_{\ell,\beta,n}(0) 
&=\left(n+\frac{\ell-1}{2}\right)^{2}\pi^{2}
-\ell(\ell+1)+s_{\ell,\beta}+l^{2}(n),~~
\beta\in\mathbb{R},\\
\lambda_{\ell,\beta,n}(q) 
&=\lambda_{\ell,\beta,n}(0)+\int_{0}^{1}q(t)dt+l^{2}(n),~~
\beta\in\mathbb{R},
\end{align*}
respectively, where $l^2(n)$ denotes the $n$-th term of an $l^2$-sequence and $s_{\ell,\beta}$ is a constant depending only on $\ell$ and $\beta$. Thus
\begin{align}\label{e2.11}
\lambda_{\ell,\beta,n}(q) 
&=\left(n+\frac{\ell-1}{2}\right)^2 \pi^2+\int_{0}^{1}q(t)dt-\ell(\ell+1)+s_{\ell,\beta}+l^{2}(n),~
\beta\in\mathbb{R}.
\end{align}

\begin{lemma}\label{l2.4}
For each $\ell\in\mathbb{N}_{0}$, the map 
$T_{\ell}:\mathbb{R}\rightarrow\mathbb{R}$, defined by $T_{\ell}(\beta)=s_{\ell,\beta}$,  is one to one and onto.
\end{lemma}

\begin{proof}
It suffices to show that 
$s_{\ell,\beta_{2}}-s_{\ell,\beta_{1}}=2(\beta_{2}-\beta_{1})$.
By \cite[(4.c)]{Carl1997} we obtain
\[
\lim_{n \rightarrow \infty}(\beta_{2}-\beta_{1})^{-1}n
\left(\sqrt{\lambda_{\ell, \beta_{2}, n+1}(0)}
-\sqrt{\lambda_{\ell, \beta_{1}, n+1}(0)}\right)=1,
\]
which together with \eqref{e2.11} implies 
\[
s_{\ell, \beta_{2}}-s_{\ell, \beta_{1}}
=\lim_{n \rightarrow \infty}
(\lambda_{\ell, \beta_{2}, n+1}(0)-\lambda_{\ell, \beta_{1}, n+1}(0))
=2(\beta_{2}-\beta_{1}),
\]
as desired.
\end{proof}

By Lemma \ref{l2.4}, for $\ell\in\mathbb{N}_{0}$, there exists a unique constant $\beta_{(\ell)}\in\mathbb{R}$ satisfying $s_{\ell,\beta}=T_{\ell}(\beta_{(\ell)})=0$. In particular, $T_{0}(0)=0$. Thus we obtain the following corollary.
\begin{corollary}\label{c2.5}
Suppose $\ell\in\mathbb{N}_{0}$ and $\{a_{n}\}_{n=1}^{\infty}$, $\{b_{n}\}_{n=1}^{\infty}$ are two sequences satisfying $a_{n}<b_{n}<a_{n+1}$ for each $n\in\mathbb{N}$. Then the followings are equivalent.
\begin{enumerate}
  
  \item [\textnormal{(1)}] There is an $L^2$-potential $q$ such that
   \begin{align*}
   \sigma(\ell, q, \beta_{(\ell)})=\{a_{n}\}_{n=1}^{\infty},
   \qquad
   \sigma(\ell, q, \infty)=\{b_{n}\}_{n=1}^{\infty}.
   \end{align*}
  
  \item [\textnormal{(2)}] The sequences satisfy the asymptotics
   \begin{align}
   a_{n}=&\left(n+\frac{\ell-1}{2}\right)^{2}\pi^{2}+C+l^{2}(n), \label{e2.13}\\
   b_{n}=&\left(n+\frac{\ell}{2}\right)^{2}\pi^{2}+C+l^{2}(n),\label{e2.12}
   \end{align}
   where $C \in \mathbb{R}$ is a constant and $l^2(n)$ denotes the $n$-th term of an $l^2$-sequence.
  
\end{enumerate}
\end{corollary}

\begin{proof}
(1)$\Rightarrow$(2) follows from \eqref{e2.9} and \eqref{e2.11}.

Conversely, it follows from \cite[Corollary~1.3]{AHM2008} that there exist $q\in L^{2}(0,1)$ and $\beta\in\mathbb{R}$ such that
\begin{align*}
\sigma(\ell, q, \beta)=\{a_{n}\}_{n=1}^{\infty},
\qquad
\sigma(\ell, q, \infty)=\{b_{n}\}_{n=1}^{\infty},
\end{align*}
and $\{a_{n}\}_{n=1}^{\infty}$, $\{b_{n}\}_{n=1}^{\infty}$ satisfy \eqref{e2.9} and \eqref{e2.11}, respectively. Thus $s_{\ell,\beta}=0$ and hence $\beta=\beta_{(\ell)}$, as desired.
\end{proof}

\subsection{The Paley-Wiener spaces}\label{s3}

Recall (cf. \cite[pp.84]{Lev1964}) that an entire function $f$ is of exponential type at most $a(a>0)$ if it satisfies the growth estimate $|f(z)|\leq Ce^{a|z|}$ for all $z\in \mathbb{C}$ and some constant $C>0$. And we say that an entire function $f$ belongs to the Paley–Wiener space $PW_{a}$ with $a>0$ if it is an entire function of exponential type at most $a$ and square integrable on $\mathbb{R}$.  
Equivalently, $f$ is the Fourier transform of an $L^2$-function supported on $(-a, a)$ (cf. \cite[pp.44-45]{Bran1968}).

A set $\Lambda\subset\mathbb{C}$ is called a uniqueness set for $PW_{a}$ $(a>0)$ if each function $f\in PW_{a}$ vanishing on $\Lambda$  must be identically zero.  The following fact follows immediately from the definition, and will be frequently used in the subsequent sections.

\begin{lemma}[\textnormal{\cite[pp. 186]{MP2005}}] \label{l3.1}
Suppose $a>0$.
The set $\Lambda$ is a uniqueness set of $PW_{a}$  if and only
if the exponential system
 $\{e^{i\lambda t}\mid\lambda\in\Lambda\}$ is complete in $L^{2}(-a,a)$.
\end{lemma}

The following two lemmas present some basic properties of the Paley-Wiener spaces.

\begin{lemma}\label{l3.2}
Suppose $\Lambda$ is a real sequence and the exponential system
$\{e^{i\lambda t}\mid\lambda\in\Lambda\}$ is not complete in $L^{2}(-a, a)$. Then there exists a nonzero function $f\in PW_a$ satisfying $f(\Lambda)=\{0\}$ and $f(\mathbb{R})\subseteq\mathbb{R}$.
\end{lemma}

\begin{proof}
Since the exponential system
$\{e^{i\lambda t}\mid\lambda\in\Lambda\}$ is not complete in $L^{2}(-a, a)$, there exists a nonzero function $G\in L^{2}(-a,a)$ such that
\begin{align*}
\int_{-a}^{a}G(x)e^{i\lambda x}dx=0,~~\textnormal{for all}~
\lambda\in\Lambda.
\end{align*}
Write
\begin{align*}
g(z)=\int_{-a}^{a}G(x)e^{izx}dx,~~h(z)=\overline{g(\overline{z})},
~~z\in\mathbb{C},
\end{align*}
then $g,h\in PW_{a}$ and $g(\Lambda)=h(\Lambda)=\{0\}$.

If $g+h\equiv0$, then $g(\mathbb{R})\subseteq i\mathbb{R}$ and write $f=ig$.
If $g+h\not\equiv0$, then write
$f=g+h$. It is clear that $f\in PW_{a}$,
$f(\Lambda)=\{0\}$ and $f(\mathbb{R})\subseteq\mathbb{R}$, which completes the proof.
\end{proof}

\begin{lemma}\label{l3.3}
Suppose $\Lambda$ is a real sequence and the exponential system
$\{e^{i\lambda t}\mid\lambda\in\Lambda\}$ is complete in $L^{2}(-a, a)$. If $f\in PW_{a}$ satisfies $f(\Lambda)\subseteq\mathbb{R}$, then $f(\mathbb{R})\subseteq\mathbb{R}$.
\end{lemma}

\begin{proof}
Write $g(z)=\overline{f(\overline{z})}$ for $z\in\mathbb{C}$ and it is clear that $g\in PW_{a}$ and $g(\Lambda)\subseteq\mathbb{R}$. Note that for every $\lambda\in\Lambda$, $f(\lambda)-g(\lambda)=0$ and then by Lemma \ref{l3.1}, $f-g\equiv 0$. This leads to $f(\mathbb{R})\subseteq\mathbb{R}$ and thus we complete the proof.
\end{proof}

\section{The connections between Paley-Wiener
spaces and potentials}\label{s8}

In this section we establish the connections between Paley-Wiener spaces and the related potentials.

\begin{proposition}\label{l3.5}
Suppose $a\in(0,1]$, $\ell\in\mathbb{N}_{0}$ and
$q$, $\tilde{q}\in L^{2}(0,1)$ satisfy $q=\tilde{q}$ on $(a,1)$.
Then 
\begin{align}\label{e3.9}
z^{2\ell+2}(\phi_{\ell}(z^2, 1, q)\phi_{\ell}^{\prime}(z^2, 1, \tilde{q})
-\phi_{\ell}(z^2, 1, \tilde{q})\phi_{\ell}^{\prime}(z^2, 1, q))
=f(z)+C
\end{align}
for some even function $f\in PW_{2a}$ and a constant $C$.
\end{proposition}

\begin{proof}
We prove the proposition in the case that $\ell$ is even and the odd case is similar.

First we show that it suffices to prove the case $a=1$. In fact, for each $a\in(0,1]$ we write $u_{a}(t)=u(at)$. Then the boundary value problem on the interval $(0,a]$:
\begin{align*}
-u^{\prime \prime}(x)+\frac{\ell(\ell+1)}{x^{2}}u(x)+q(x)u(x)
=\lambda u(x),
\end{align*}
with the boundary conditions
\[
\lim_{x \rightarrow 0}\frac{u(x)}{x^{\ell+1}}<\infty, 
~~u^{\prime}(a)+\beta u(a)=0,
\] 
can be rewritten as the boundary value problem on the interval $(0,1]$:
\begin{align*}
-u_{a}^{\prime \prime}(t)+\frac{\ell(\ell+1)}{t^{2}}u_{a}(t)+a^{2}q(at)u_{a}(t)
=a^{2}\lambda u_{a}(t),
\end{align*}
with the boundary conditions
\begin{align*}
\lim_{t\rightarrow 0}\frac{u_{a}(t)}{t^{\ell+1}}<\infty, 
~~u_{a}^{\prime}(1)+\beta au_{a}(1)=0.   
\end{align*}
Write $q_{a}(t)=a^{2}q(at)$ for $t\in(0,1]$ and we have
\begin{align*}
\phi_{\ell}(z^2, at, q)=u(at)=u_{a}(t)=\phi_{\ell}(a^{2}z^2,t, q_{a}),
\end{align*}
which implies
\begin{align}\label{e3.0}
\phi_{\ell}(z^2, a, q)=\phi_{\ell}(a^{2}z^2,1, q_{a}),~~
\phi_{\ell}^{\prime}(z^2, a, q)=\frac{\phi_{\ell}^{\prime}(a^{2}z^2,1, q_{a})}{a}.
\end{align}

It follows from the proof of \cite[Lemma 3.2]{XYB2023} that
\begin{align}\label{e3.1}
&\int_{0}^{t} (q-\tilde{q})\phi_{\ell}(z^2, x, q) \phi_{\ell}(z^2, x, \tilde{q})dx \nonumber \\
&\qquad\qquad\qquad\qquad
=\phi_{\ell}(z^2,t, q)\phi_{\ell}^{\prime}(z^2,t, \tilde{q})
-\phi_{\ell}(z^2, t, \tilde{q})\phi_{\ell}^{\prime}(z^2, t, q)
\end{align}
for any $t\in[0,1]$. Thus \eqref{e3.0} and \eqref{e3.1} yield
\begin{align*}
f(z)+C
=&z^{2\ell+2}(\phi_{\ell}(z^2, 1, q)\phi_{\ell}^{\prime}(z^2, 1, \tilde{q})
-\phi_{\ell}(z^2, 1, \tilde{q})\phi_{\ell}^{\prime}(z^2, 1, q))\\
=&z^{2\ell+2}(\phi_{\ell}(z^2,a,q)\phi_{\ell}^{\prime}(z^2,a,\tilde{q})
-\phi_{\ell}(z^2, a, \tilde{q})\phi_{\ell}^{\prime}(z^2, a, q))\\
=&\frac{z^{2\ell+2}}{a}(\phi_{\ell}(a^{2}z^2,1,q_{a})
\phi_{\ell}^{\prime}(a^{2}z^2,1,\tilde{q}_{a})
-\phi_{\ell}(a^{2}z^2,1, \tilde{q}_{a})
\phi_{\ell}^{\prime}(a^{2}z^2,1, q_{a})).
\end{align*}
Let $w=az$ and write $h(w)=f(w/a)$. To prove \eqref{e3.9}, it suffices to show that 
\[
\frac{w^{2\ell+2}}{a^{2\ell+3}}(\phi_{\ell}(w^{2},1,q_{a})
\phi_{\ell}^{\prime}(w^2,1,\tilde{q}_{a})
-\phi_{\ell}(w^2,1, \tilde{q}_{a})
\phi_{\ell}^{\prime}(w^2,1, q_{a}))=h(w)+C
\]
for some $h\in PW_{2}$ and a constant $C$.

Now assume $a=1$. If $a_{1}$, $a_{2}$, $\cdots$, $a_{\ell/2},b_{1}$, $b_{2}$, $\cdots$, $b_{\ell/2}\in\mathbb{R}$ satisfy
\begin{align*}
a_{1}<b_{1}<a_{2}<b_{2}<\cdots<a_{\ell/2}<b_{\ell/2}
<\min\{\lambda_{\ell, \beta_{(\ell)}, 1}(q),\lambda_{\ell, \beta_{(\ell)}, 1}(\tilde{q})\},
\end{align*}
then by Corollary \ref{c2.5}, there exists $q_{0}\in L^{2}(0,1)$ such that
\begin{align*}
\sigma(0,q_{0},0)
&=\sigma(\ell, q, \beta_{(\ell)})\cup\{a_{1},a_{2},\ldots,a_{\ell/2}\},\\
\sigma(0, q_{0}, \infty)
&=\sigma(\ell, q, \infty)\cup\{b_{1},b_{2},\ldots,b_{\ell/2}\},
\end{align*}
which together with Proposition \ref{p2.2} implies
\begin{align*}
&(\phi_{\ell}^{\prime}(\lambda, 1, q)
+\beta_{(\ell)}\phi_{\ell}(\lambda, 1, q))
\prod_{k=1}^{\frac{\ell}{2}}(a_{k}-\lambda)\\
=&\frac{1}{C_{\ell-1}}\prod_{n=1}^{\infty}
\frac{\lambda_{\ell,\beta_{(\ell)},n}(q)-\lambda}{(n+\frac{\ell-1}{2})^{2}\pi^{2}}
\prod_{k=1}^{\frac{\ell}{2}}(a_{k}-\lambda)\\
=&\prod_{k=\frac{\ell}{2}+1}^{\infty}
\frac{\lambda_{\ell,\beta_{(\ell)},k-\frac{\ell}{2}}(q)-\lambda}
{(k-\frac{1}{2})^{2}\pi^{2}}
\prod_{k=1}^{\frac{\ell}{2}}\frac{a_{k}-\lambda}{(k-\frac{1}{2})^{2}\pi^{2}}\\
=&\phi_{0}^{\prime}(\lambda, 1, q_{0}),
\end{align*}
and
\begin{align*}
\phi_{\ell}(\lambda, 1, q)
\prod_{k=1}^{\frac{\ell}{2}}(b_{k}-\lambda)
=\phi_{0}(\lambda, 1, q_{0}).
\end{align*}
Similarly, there exists $\tilde{q}_{0}\in L^{2}(0,1)$ such that
\begin{align*}
(\phi_{\ell}^{\prime}(\lambda, 1, \tilde{q})
+\beta_{(\ell)}\phi_{\ell}(\lambda, 1, \tilde{q}))
\prod_{k=1}^{\frac{\ell}{2}}(a_{k}-\lambda)
=\phi_{0}^{\prime}(\lambda, 1, \tilde{q}_{0}),\\
\phi_{\ell}(\lambda, 1, \tilde{q})
\prod_{k=1}^{\frac{\ell}{2}}(b_{k}-\lambda)
=\phi_{0}(\lambda, 1, \tilde{q}_{0}).
\end{align*}
Therefore, it holds 
\begin{align*}
F(z):=&(\phi_{\ell}(z^2, 1, q)\phi_{\ell}^{\prime}(z^2, 1, \tilde{q})
-\phi_{\ell}(z^2, 1, \tilde{q})\phi_{\ell}^{\prime}(z^2, 1, q))
\prod_{k=1}^{\frac{\ell}{2}}(a_{k}-z^2)(b_{k}-z^2)\\
=&(\phi_{\ell}(z^2, 1, q)(\phi_{\ell}^{\prime}(z^2, 1, \tilde{q})+\beta_{(\ell)}\phi_{\ell}(z^2, 1, \tilde{q})))\\
&-\phi_{\ell}(z^2, 1, \tilde{q})(\phi_{\ell}^{\prime}(z^2, 1, q)+\beta_{(\ell)}\phi_{\ell}(z^2, 1,q))))
\prod_{k=1}^{\frac{\ell}{2}}(a_{k}-z^2)(b_{k}-z^2)\\
=&\phi_{0}(z^2, 1, q_{0})\phi_{0}^{\prime}(z^2, 1, \tilde{q}_{0})-\phi_{0}(z^2, 1, \tilde{q}_{0})\phi_{0}^{\prime}(z^2, 1, q_{0}).
\end{align*}
Recall (cf. \cite[Theorem 1.1(i)]{Koro2004} or \cite[Theorem 1(1)]{BBP2017}) that there exists an even function $g\in PW_{2}$ and a constant $C$ such that
\begin{align*}
z^{2}\left(\phi_{0}(z^{2}, 1, q_{0})\phi_{0}^{\prime}(z^{2}, 1, \tilde{q}_{0})-\phi_{0}
(z^{2}, 1, \tilde{q}_{0})\phi_{0}^{\prime}(z^{2}, 1, q_{0})\right)
=g(z)+C
\end{align*}
and hence $z^{2}F(z)=g(z)+C$.
This implies
\begin{align*}
&z^{2\ell+2}(\phi_{\ell}(z^2, 1, q)\phi_{\ell}^{\prime}(z^2, 1, \tilde{q})
-\phi_{\ell}(z^2, 1, \tilde{q})\phi_{\ell}^{\prime}(z^2, 1, q))\\
=&z^{2\ell+2}F(z)\prod_{k=1}^{\frac{\ell}{2}}\frac{1}{(a_{k}-z^2)(b_{k}-z^2)}\\
=&(g(z)+C)\prod_{k=1}^{\frac{\ell}{2}}\frac{z^{4}}{(a_{k}-z^2)(b_{k}-z^2)}
-C+C.
\end{align*}

Write 
\begin{align*}
f(z)=(g(z)+C)\prod_{k=1}^{\frac{\ell}{2}}\frac{z^{4}}{(a_{k}-z^2)(b_{k}-z^2)}-C,
\end{align*}
and it suffices to prove that $f\in PW_{2}$.

It is clear that $f$ is an entire function of exponential type at most $2$. Moreover, a simple computation gives 
\begin{align*}
|f(z)|=&\left|g(z)\prod_{k=1}^{\frac{\ell}{2}}
\frac{z^{4}}{(a_{k}-z^2)(b_{k}-z^2)}
+C\frac{\prod_{k=1}^{\frac{\ell}{2}}z^{4}
-\prod_{k=1}^{\frac{\ell}{2}}(a_{k}-z^2)(b_{k}-z^2)}
{\prod_{k=1}^{\frac{\ell}{2}}(a_{k}-z^2)(b_{k}-z^2)}\right|\\
=&O(|g(z)|)+\frac{O(|z|^{2\ell-2})}{|z|^{2\ell}+O(|z|^{2\ell-2})}\\
=&O(|g(z)|)+O\left(\frac{1}{|z|^{2}}\right),
\end{align*}
as $|z|\rightarrow\infty$,
which implies that $f$ is square integrable on $\mathbb{R}$ and hence $f\in PW_{2}$, as desired.
\end{proof}

Roughly speaking, due to Corollary \ref{c2.5}, the proposition below provides the converse of Proposition \ref{l3.5}.

\begin{proposition}\label{l3.6}
Suppose $\ell\in\mathbb{N}_{0}$, $a\in(0,1]$ and $\{a_{1},\cdots, a_{\ell}\}$ with $a_{k}\in(0,\infty)$.
Then there exists $\varepsilon>0$ such that for any even function $f \in PW_{2a}$ which is real on $\mathbb{R}$ with $\|f\|_{2}<\varepsilon$, 
and $f(\pm a_{k})=f(0)$ for $k=1$, $\cdots$, $\ell$, 
there exists a potential
function $q\in L^{2}(0,1)$ satisfying $q=0$ on $(a,1)$, 
such that the following hold:
\begin{enumerate}

 \item [\textnormal{(1)}] If $\ell$ is even, then
    \begin{align}\label{e3.2}
    zH(z)\left(\frac{A(z)\cos(az-\frac{\ell\pi}{2})}{H_{e1}(az)}
    -\frac{B(z)\sin(az-\frac{\ell\pi}{2})}{H_{e2}(az)}\right) 
    =f(z)-f(0).
    \end{align}

 \item [\textnormal{(2)}] If $\ell$ is odd, then
    \begin{align}\label{e3.3}
    H(z)\left(\frac{A(z)\cos(az-\frac{\ell\pi}{2})}{H_{o1}(az)}
    -\frac{z^{2}B(z)\sin(az-\frac{\ell\pi}{2})}
    {H_{o2}(az)}\right) 
    =f(z)-f(0),
    \end{align}
    
\end{enumerate}
where
\begin{gather*}
H(z)=\prod_{k=1}^{\ell}
(z^2-a_{k}^{2}),\\
A(z)=z\phi_{\ell}(z^2, a, q),~~
B(z)=\phi_{\ell}^{\prime}(z^{2},a,q)
+\frac{\beta_{(\ell)}}{a}\phi_{\ell}(z^{2},a,q),\\
H_{e1}(z)=\prod_{k=1}^{\frac{\ell}{2}}
\left(z^2-\left(k-\frac{1}{2}\right)^{2} \pi^{2}\right),~~
H_{e2}(z)=\prod_{k=1}^{\frac{\ell}{2}}
(z^{2}-k^{2} \pi^2),\\
H_{o1}(z)=\prod_{k=1}^{\frac{\ell-1}{2}}
\left(z^2-\left(k-\frac{1}{2}\right)^2\pi^2\right),~~
H_{o2}(z)=\prod_{k=1}^{\frac{\ell+1}{2}}
(z^2-k^2\pi^2).
\end{gather*}
\end{proposition}

\begin{proof}
We only prove the proposition in the case that $\ell$ is even and the odd case is similar.

We split the proof into two cases: $a=1$ and $a\neq1$.

\textbf{Case 1}: $a=1$. The proof consists of the following four steps.

\textbf{Step 1}: We first prove the result under the assumption that
$\{a_{1},\cdots,a_{\ell}\}=\{k\pi, (k-1/2)\pi\mid k=1,\cdots,\ell/2\}$.
Indeed, write 
\begin{align*}
F(z)=\frac{f(z)-f(0)}{I(z)}+f(0),
\end{align*}
where
\begin{align*}
I(z)=\frac{H(z)}{\prod_{k=1}^{\frac{\ell}{2}}
\left(z^2-\left(k-\frac{1}{2}\right)^{2} \pi^{2}\right)
\prod_{k=1}^{\frac{\ell}{2}}
(z^{2}-k^{2} \pi^2)}.
\end{align*}
It is clear that $F \in PW_{2}$ is an even function, which is real on $\mathbb{R}$. It only needs to show that $\|f\|_{2}\leq c\|F\|_{2}$ for some positive constant $c$ independent of $f$. 

In fact, let $r=\max\{a_{1},a_{2},\ldots,a_{\ell},\ell\pi/2\}+1$ and $S_{r}=\{z\in\mathbb{C}\mid|z|\leq r\}$. Then the maximum modulus theorem shows that there exists $\zeta\in\partial S_{r}$ such that
$\sup_{z\in S_{r}}|f(z)|\leq|f(\zeta)|$ and hence
\begin{align*}
\|f\|_{2}^{2}
=&\int_{\mathbb{R}}|f(x)|^{2}dx\\
\leq& 2r|f(\zeta)|^{2}+\int_{\mathbb{R}\setminus(-r,r)}|f(x)|^{2}dx\\
=&2r|(F(\zeta)-F(0))I(\zeta)+F(0)|^{2}
+\int_{\mathbb{R}\setminus(-r,r)}|(F(x)-F(0))I(x)+F(0)|^{2}dx\\
\leq&4r|F(\zeta)|^{2}
\max_{z\in\partial S_{r}}|I(z)|^{2}
+4r|F(0)|^{2}
\max_{z\in\partial S_{r}}|1-I(z)|^{2}\\
&+2\int_{\mathbb{R}\setminus(-r,r)}
|F(x)I(x)|^{2}dx 
+2|F(0)|^{2}\int_{\mathbb{R}\setminus(-r,r)}
|1-I(x)|^{2}dx. 
\end{align*}
Since $I$ is bounded on $\mathbb{R}\setminus(-r,r)$, and
$1-I$ is square integrable on $\mathbb{R}\setminus(-r,r)$, then
\begin{align*}
c_{1}&=\max\left\{
4r\max_{z\in\partial S_{r}}|I(z)|^{2},~
4r\max_{z\in\partial S_{r}}|1-I(z)|^{2},\right.\\
&\left.\qquad\qquad
2\sup_{x\in\mathbb{R}\setminus(-r,r)}|I(x)|^{2},~
2\int_{\mathbb{R}\setminus(-r,r)}|1-I(x)|^{2}dx
\right\}<\infty,
\end{align*}
which leads to
\begin{align*}
\|f\|_{2}^{2}
\leq & c_{1}\left(|F(\zeta)|^{2}+2|F(0)|^{2}
+\int_{\mathbb{R}\setminus(-r,r)}|F(x)|^{2}dx\right)\\
\leq& c_{1}(|F(\zeta)|^{2}+2|F(0)|^{2}+\|F\|_{2}^{2}).
\end{align*}

Recall that for $F\in PW_{2}$, it is the Fourier transform of an $L^2$-function $\hat{F}$ supported on $(-2, 2)$, that is, 
\begin{align*}
F(z)=\int_{-2}^{2}\hat{F}(x)e^{izx}dx,~~z\in\mathbb{C}.
\end{align*}
Then H\"{o}lder's inequality and Plancherel's theorem 
\cite[Theorem 9.13]{Rud1987} together yield 
\begin{align}\label{e3.10}
|F(z)|
\leq&\|\hat{F}\|_{2}\|\chi_{(-2,2)}e^{izx}\|_{2}
=\frac{\|F\|_{2}\|\chi_{(-2,2)}e^{izx}\|_{2}}{\sqrt{2\pi}}
\leq\frac{\sqrt{2}e^{2r}\|F\|_{2}}{\sqrt{\pi}}
\end{align}
for all $z\in S_{r}$. Therefore,
\begin{align*}
\|f\|_{2}\leq c_{1}\left(3\frac{\sqrt{2}e^{2r}}{\sqrt{\pi}}+1\right)\|F\|_{2}.
\end{align*}

\textbf{Step 2}: Now let even function  $f \in PW_{2}$ with $f(\pm k\pi)=f(\pm(k-1/2)\pi)=f(0)$ for $k=1,\cdots,\ell/2$ and $\|f\|_{2}<\varepsilon$ (where the choice of $\varepsilon$ will be specified later). 
We will construct two functions $A_{1},B_{1}$ satisfying
\begin{align}\label{e3.8}
zH(z)\left(\frac{A_{1}(z)\cos(z-\frac{\ell\pi}{2})}{H_{e1}(z)}
-\frac{B_{1}(z)\sin(z-\frac{\ell\pi}{2})}{H_{e2}(z)}\right) 
=f(z)-f(0).
\end{align}

First observe that it is needed that
\begin{align}\label{e3.71}
(A_{1}H_{e2})(\pm k\pi)=0,~~\textnormal{for}~k=1,\cdots,\ell/2,
\end{align}
and
\begin{align*}
(B_{1}H_{e1})\left(\pm\left(k-\frac{1}{2}\right)\pi\right)=0,
~~\textnormal{for}~k=1,\cdots,\ell/2.
\end{align*}
The facts that $H = H_{e1}H_{e2}$, $f$ is an even function and \eqref{e3.8} together show 
\begin{gather}\label{e3.4}
\begin{split}
&(A_{1}H_{e2})(0)=f'(0)=0,\\
(A_{1}H_{e2})\left(\pm\left(n+\frac{\ell}{2}\right)\pi\right)
=&(-1)^{n} \frac{f((n+\frac{\ell}{2})\pi)-f(0)}
{\pm(n+\frac{\ell}{2})\pi}:=(-1)^{n+\frac{\ell}{2}}s_{n+\frac{\ell}{2}},
\end{split}
\end{gather}
and
\begin{align*}
(B_{1}H_{e1})\left(\pm\left(n+\frac{\ell-1}{2}\right)\pi\right)
=(-1)^{n+1} 
\frac{f((n+\frac{\ell-1}{2})\pi)-f(0)}
{\pm(n+\frac{\ell-1}{2})\pi}.
\end{align*}
Those together with the fact that $f\in PW_{2}$ imply the sequences $\{(A_{1}H_{e2})(j\pi)\}_{j\in\mathbb{Z}}$ and $\{(B_{1}H_{e1})((j-1/2)\pi)\}_{j\in\mathbb{Z}}$ belong to $l^{2}$, equivalently, it holds
\begin{align*}
\sum_{j=-\infty}^{\infty}(A_{1}H_{e2})(j\pi)e^{-ij\pi x}~~
\textnormal{and}~~
\sum_{j=-\infty}^{\infty}(B_{1}H_{e1})\left(\left(j-\frac{1}{2}\right)\pi\right)
e^{-i(j-\frac{1}{2})\pi x}
\end{align*}
belong to $L^{2}(-1,1)$. Define
\begin{align}
h(z)=&\int_{-1}^{1}\frac{1}{2}
\sum_{j=-\infty}^{\infty}(A_{1}H_{e2})(j\pi)e^{-ij\pi x}
\cdot e^{izx}dx,\label{e3.6}\\
g(z)=&\int_{-1}^{1}\frac{1}{2}
\sum_{j=-\infty}^{\infty}(B_{1}H_{e1})
\left(\left(j-\frac{1}{2}\right)\pi\right)e^{-i(j-\frac{1}{2})\pi x}
\cdot e^{izx}dx.\nonumber
\end{align}
Then $h,g\in PW_{1}$ and they are real on $\mathbb{R}$ by Lemma \ref{l3.3}. Moreover,

\begin{align*}
h(j\pi)=(A_{1}H_{e2})(j\pi),~~
g\left(\left(j-\frac{1}{2}\right)\pi\right)
=(B_{1}H_{e1})\left(\left(j-\frac{1}{2}\right)\pi\right),~~j\in\mathbb{Z}.
\end{align*}

Now define
\begin{align}\label{e3.12}
(A_{1}H_{e2})(z)
=\sin \left(z-\frac{\ell \pi}{2}\right)+h(z),~
(B_{1}H_{e1})(z)
=\cos\left(z-\frac{\ell \pi}{2}\right)+g(z).
\end{align}
A simple computation shows
\begin{align*}
Q(z)=& z\left((A_{1}H_{e2})(z)\cos\left(z-\frac{\ell\pi}{2}\right)
-(B_{1}H_{e1})(z)\sin\left(z-\frac{\ell\pi}{2}\right)\right) \\
=& z\left(h(z)\cos\left(z-\frac{\ell\pi}{2}\right)
+g(z)\sin\left(z-\frac{\ell\pi}{2}\right)\right),
\end{align*}
and
$Q/z\in PW_{2}$ coincides with $(f-f(0))/z$ at $\{j\pi/2\}_{j\in\mathbb{Z}}$.
Recall that $\{e^{ij\pi t/2}\mid j\in\mathbb{Z}\}$ is complete in $L^{2}(-2,2)$, so again from Lemma \ref{l3.1}, it follows that $Q(z)=f(z)-f(0)$.
Thus, the constructed functions $A_{1}$ and $B_{1}$ satisfy \eqref{e3.8}.

\textbf{Step 3}: We will show that the square of the zeros of $A_{1}/z$ and $B_{1}$ satisfy \eqref{e2.9} and \eqref{e2.11}, respectively and are interlacing.

We investigate the asymptotics of the zeros of
$(A_{1}H_{e2})(z)$ and those of $(B_{1}H_{e1})(z)$ follow in a similar way. 
First observe 
\begin{align}\label{e3.11}
h(z)=\sum_{|j|>\frac{\ell}{2}}
\frac{\sin z}{z-j\pi}s_{j},
~~z\in\mathbb{C}.
\end{align}
Moreover, we have
\begin{align*}
\frac{\sin z}{z-j\pi}
=(-1)^{j}\frac{\sin(z-j\pi)}{z-j\pi}
=\frac{(-1)^{j}}{2}\int_{-1}^{1}e^{-ij\pi x}\cdot e^{izx}dx,
~~z\in\mathbb{C},
\end{align*}
which together with \eqref{e3.71}-\eqref{e3.6} and H\"{o}lder's inequality implies, for any fixed $z\in\mathbb{C}$, it holds
\begin{align*}
&\left|h(z)-\sum_{|j|>\frac{\ell}{2}}^{|j|\leq N}
\frac{\sin z}{z-j\pi}s_{j}\right|\\
=&\left|\int_{-1}^{1}\frac{1}{2}
\sum_{|j|>N}(A_{1}H_{e2})(j\pi)e^{-ij\pi x}
\cdot e^{izx}dx\right|\\
\leq&\sqrt{2}e^{|\textnormal{Im}z|}
\left\|\frac{1}{2}
\sum_{|j|>N}(A_{1}H_{e2})(j\pi)e^{-ij\pi x}\right\|_{L^{2}(-1,1)}
\longrightarrow 0,
~~\textnormal{as}~N\rightarrow\infty.
\end{align*}
Furthermore, $h$ is also an odd function since $s_{j}=-s_{-j}$ for $|j|>\ell/2$. 

Next, we will prove that one can choose $\varepsilon$ small enough such that $|h(z)|<1/4$ for $|\textnormal{Im}z|\leq\pi/6$ and $\|f\|_{2}<\varepsilon$. 
Similar to \eqref{e3.10}, we have
$|f(x)|\leq \sqrt{2}\|f\|_{2}/\sqrt{\pi}$ for all $x\in\mathbb{R}$, which together with \eqref{e3.4} and \eqref{e3.11} implies that 
\begin{align*}
|h(z)|\leq\frac{2\sqrt{2}}{\sqrt{\pi}}\|f\|_{2}\sum_{|j|>\frac{\ell}{2}}\left|
\frac{\sin z}{(z-j\pi)j\pi}
\right|.
\end{align*}
If $|\textnormal{Im}z|\leq\pi/6$ and $|\textnormal{Re}z-k\pi|\leq\pi/2$ for fixed $k\in\mathbb{Z}$, then $|\sin z|\leq e^{\pi/6}$. In addition, again by H\"{o}lder's inequality and \cite[(9), pp.187]{Ahl1978}, we have
\begin{align*}
\sum_{j\neq k}
\left|
\frac{\sin z}{(z-j\pi)j\pi}
\right|
\leq&
\sum_{j > k}\left|
\frac{e^{\frac{\pi}{6}}}{((j-k-\frac{1}{2})\pi)j\pi}\right|
+\sum_{j < k}\left|
\frac{e^{\frac{\pi}{6}}}{((k-\frac{1}{2}-j)\pi)j\pi}\right|\\
\leq&\frac{2e^{\frac{\pi}{6}}}{\pi^{2}}
\left(
\sum_{j\in\mathbb{Z}}\frac{1}{(j-\frac{1}{2})^{2}}
\right)^{1/2}
\left(
\sum_{j\neq 0}\frac{1}{j^{2}}
\right)^{1/2}=\frac{2e^{\frac{\pi}{6}}}{\sqrt{3}}.
\end{align*}
Moreover,
\begin{align*}
\left|
\frac{\sin z}{(z-k\pi)k\pi}
\right|
\leq
\left|
\frac{\sin(z-k\pi)}{(z-k\pi)}
\right|
\leq
\sup_{\substack{|\textnormal{Im}z|\leq\pi/6\\|\textnormal{Re}z|\leq\pi/2}}
\left|
\frac{\sin z}{z}
\right|\leq 6e^{\frac{\pi}{6}},
~~k\neq0.
\end{align*}

Consequently, for $|\textnormal{Im}z|\leq\pi/6$, we get
\begin{align*}
|h(z)|\leq \frac{16\sqrt{2}}{\sqrt{\pi}}e^{\frac{\pi}{6}}\|f\|_{2}.
\end{align*}
Thus we need only to choose any positive $\varepsilon$ such that $\varepsilon<\sqrt{\pi}/(64\sqrt{2}e^{\pi/6})$. 

Now we determine the zero set of $A_{1}/z$.
Note that
$\inf_{|z-j\pi|=\pi/6}|\sin z|=1/2$ for all $j\in \mathbb{Z}$ and it follows from the Rouch\'{e} theorem that
\begin{align*}
(A_{1}H_{e2})(z)
=\sin\left(z-\frac{\ell\pi}{2}\right)+h
\end{align*}
has a unique zero in the disc $\{z\in\mathbb{C}\mid |z-j\pi|<\pi/6\}$. 
Moreover, the fact that $h$ is real on $\mathbb{R}$ and odd implies that for all $j\in\mathbb{Z}$, 
\begin{align*}
&(A_{1}H_{e2})\left(j\pi+\frac{\pi}{6}\right)
(A_{1}H_{e2})\left(j\pi-\frac{\pi}{6}\right)\\
=&\left(\frac{(-1)^{j-\frac{\ell}{2}}}{2}+h\left(j\pi+\frac{\pi}{6}\right)\right)
\left(-\frac{(-1)^{j-\frac{\ell}{2}}}{2}+h\left(j\pi-\frac{\pi}{6}\right)\right)\\
<&0,
\end{align*}
which shows that the unique zero of $A_{1}H_{e2}$, lying in the disc $\{z\in\mathbb{C}\mid |z-j\pi|<\pi/6\}$, belongs to $(j\pi-\pi/6,j\pi+\pi/6)$ and the all zeros of $A_{1}H_{e2}$ are symmetric about the origin.
In addition, the zero set of $H_{e2}$ is $\{k\pi\mid k=\pm 1,\cdots,\pm\ell/2\}$, so let $\lambda_{n}$ be the zero of $A_{1}/z$ lying in $((n+\ell/2)\pi-\pi/6, (n+\ell/2)\pi+\pi/6)$.

Next we will prove that $\{\pm\lambda_{n}\}_{n\in\mathbb{N}}$  is exactly the zero set of $A_{1}/z$ and $\lambda_{n}$ satisfies    
\begin{align}\label{e3.5}
\lambda_{n}=\left(n+\frac{\ell}{2}\right)\pi+\frac{C}{n}+\frac{l^{2}(n)}{n},
~~C=-\frac{f(0)}{\pi}.
\end{align}

Since $f\in PW_{2}$, there exists $\hat{f}\in L^{2}(-2,2)$ such that 
\begin{align*}
f(z)=\int_{-2}^{2}\hat{f}(x)e^{izx}dx.
\end{align*}  
which implies the Fourier coefficients of $\hat{f}$ are $\{f(j\pi/2)\}_{j\in\mathbb{Z}}$ and hence $\{f(j\pi/2)\}_{j\in\mathbb{Z}}\in l^{2}$. In particular, $\{f((n+\ell/2)\pi)\}_{n\in\mathbb{N}}\in l^{2}$.

If $f(0)\neq0$, then it follows from \eqref{e3.4} that $\lambda_{n}\neq(n+\ell/2)\pi$ for sufficiently large $n\in\mathbb{N}$.
Thus for sufficiently large $n\in\mathbb{N}$, it holds
\begin{align*}
0=\frac{(A_{1}H_{e2})(\lambda_{n})}{\sin(\lambda_{n}-\frac{\ell}{2}\pi)}
=\frac{\sin(\lambda_{n}-\frac{\ell}{2}\pi)+h(\lambda_{n})}
{\sin(\lambda_{n}-\frac{\ell}{2}\pi)}
=1+\sum_{|j|>\frac{\ell}{2}}\frac{(-1)^{\frac{\ell}{2}}}{\lambda_{n}-j\pi}
s_{j},
\end{align*}
whence
\begin{align}\label{e3.13}
\lambda_{n}=&\left(n+\frac{\ell}{2}\right)\pi
+(-1)^{\frac{\ell}{2}}s_{n+\frac{\ell}{2}}
\left(1+\sum_{\substack{|j|>\frac{\ell}{2}\\j\neq n+\frac{\ell}{2}}} \frac{(-1)^{\frac{\ell}{2}}}{\lambda_{n}-j\pi}s_{j}\right)^{-1}.
\end{align}
A simple computation gives
\begin{align*}
\left|\sum_{\substack{|j|>\frac{\ell}{2}\\j\neq n+\frac{\ell}{2}}} \frac{(-1)^{\frac{\ell}{2}}}{\lambda_{n}-j\pi}s_{j}\right|
\leq&
\sup_{\substack{|j|>\frac{\ell}{2}\\j\neq n+\frac{\ell}{2}}} 
\left|
\frac{(n+\frac{\ell}{2}-j)\pi}{\lambda_{n}-j\pi}
\right|
\left|\sum_{\substack{|j|>\frac{\ell}{2}\\j\neq n+\frac{\ell}{2}}} \frac{s_{j}}{(n+\frac{\ell}{2}-j)\pi}\right|\\
\leq&
\frac{6}{5}
\left|\sum_{\substack{|j|>\frac{\ell}{2}\\j\neq n+\frac{\ell}{2}}} \frac{s_{j}}{(n+\frac{\ell}{2}-j)\pi}\right|.
\end{align*}
Again by \eqref{e3.4}, $\{s_{j}\}_{|j|>\ell/2}\in l^{2}$ and then by the boundedness of the discrete Hilbert transform on $l^{2}$ and the above inequality, we have 
\begin{align*}
\left\{\sum_{\substack{|j|>\frac{\ell}{2}\\j\neq n+\frac{\ell}{2}}} \frac{s_{j}}{(n+\frac{\ell}{2}-j)\pi}\right\}_{n\in\mathbb{N}}
\in l^{2},
~~~\textnormal{and hence}~~
\left\{
\sum_{\substack{|j|>\frac{\ell}{2}\\j\neq n+\frac{\ell}{2}}} \frac{(-1)^{\frac{\ell}{2}}}{\lambda_{n}-j\pi}s_{j}\right\}_{n\in\mathbb{N}}
\in l^{2},
\end{align*}
which combined with \eqref{e3.13} yields
\begin{align*}
\lambda_{n}
=&\left(n+\frac{\ell}{2}\right)\pi
+\left(-\frac{f(0)}{(n+\frac{\ell}{2})\pi}
+\frac{f((n+\frac{\ell}{2})\pi)}{(n+\frac{\ell}{2})\pi}\right)(1+l^{2}(n))^{-1}\\
=&\left(n+\frac{\ell}{2}\right)\pi+\frac{C}{n}+\frac{l^{2}(n)}{n},
\end{align*}
where $C=-f(0)/\pi$.

If $f(0)=0$, let 
$\mathbb{M}=\{n\in\mathbb{N}\mid\lambda_{n}=(n+\ell/2)\pi\}$. If $n\in\mathbb{M}$, then $\lambda_{n}$ satisfies \eqref{e3.5}. If $n\notin\mathbb{M}$, then as in the case $f(0)\neq 0$, $\lambda_{n}$ also satisfies \eqref{e3.5}.

Now it suffices to show that $A_{1}/z$ has no other zeros. Indeed, write 
\begin{align}\label{e3.14}
(A_{1}H_{e2})(z)
=zP
\prod_{k=1}^{\infty}\left(1-\frac{z^{2}}{\lambda_{k}^{2}}\right)
\prod_{k=1}^{\frac{\ell}{2}}(z^{2}-(k\pi)^{2})
\end{align}
for some entire function $P$.  Then it follows from \cite[Lemma 1]{BBP2017} that $P$ is a constant, as desired.

Similarly, the zeros $\mu_{n}$ of $B_{1}$ satisfy $\mu_{n} \in((n+\ell/2)\pi-2\pi/3, (n+\ell/2)\pi-\pi/3)$ and 
\begin{align*}
\mu_{n}=\left(n+\frac{(\ell-1)}{2}\right)\pi+\frac{C}{n}+\frac{l^{2}(n)}{n}
\end{align*}
for some sufficiently small $\varepsilon>0$.

Consequently, $\lambda_{n}^{2}$ and $\mu_{n}^{2}$ satisfy \eqref{e2.9} and \eqref{e2.11}, respectively and are interlacing.

\textbf{Step 4}: We first find $q\in L^{2}(0,1)$ and then proceed to prove $A(z)=A_{1}(z)$ and $B(z)=B_{1}(z)$.

First by Corollary \ref{c2.5} and Step 3, there exists $q \in L^2(0,1)$ such that $\lambda_{\ell, \infty, n}(q)=\lambda_{n}^{2}$ and $\lambda_{\ell, \beta, n}(q)=\mu_{n}^{2}$. Then combine \eqref{e2.5} and \eqref{e3.14}, it follows that 
$A_{1}(z)=c_{2}A(z)$ for some constant $c_{2}$. 
And similarly, we have
$B_{1}(z)=c_{3}B(z)$
for some constant $c_{3}$. Note that \eqref{e2.1} gives
\begin{align*}
c_{2}=&\frac{A_{1}(z)}{A(z)} \\
=&z^{\ell}\left(\prod_{k=1}^{\frac{\ell}{2}}(z^{2}-(k\pi)^{2})\right)^{-1}\\
&\cdot\left(\sin\left(z-\frac{\ell\pi}{2}\right)+h(z)\right)
\left(\sin \left(z-\frac{\ell \pi}{2}\right)
+O\left(|z|^{-1} e^{| \textnormal{Im}(z) \mid}\right)\right)^{-1}.
\end{align*}
Take $z=(n+1/2)\pi$ and we have $c_{2}=1$ as $n\rightarrow\infty$.
Similarly, we have $c_{3}=1$. This completes the proof of the case $a=1$.

\textbf{Case 2}: $a<1$. By \eqref{e3.0}, a simple computation gives
\begin{align*}
f(z)-f(0)
=&zH(z)\left(\frac{z\phi_{\ell}(z^{2},a,q)
\cos(az-\frac{\ell\pi}{2})}{H_{e1}(az)}
\right.\\
&-\left.\frac{(\phi_{\ell}^{\prime}(z^{2},a,q)
+\frac{\beta_{(\ell)}}{a}\phi_{\ell}(z^{2},a,q))
\sin(az-\frac{\ell\pi}{2})}{H_{e2}(az)}\right) \\
=&zH(z)\left(\frac{z\phi_{\ell}(a^{2}z^{2},1,q_{a})
\cos(az-\frac{\ell\pi}{2})}{H_{e1}(az)}
\right.\\
&-\left.\frac{(\phi_{\ell}^{\prime}(a^{2}z^{2},1,q_{a})
+\beta_{(\ell)}\phi_{\ell}(a^{2}z^{2},1,q_{a})
\sin(az-\frac{\ell\pi}{2})}{aH_{e2}(az)}\right).
\end{align*}
Let $w=az$ and $g(w)=a^{2\ell+2}f(z)$. Then 
\eqref{e3.2} holds for some $q\in L^2(0,1)$ if and only if 
the following holds for some $q_{a}\in L^2(0,1)$,
\begin{align}\label{e3.7}
\begin{split}
g(w)-g(0)
=&wH_{a}(w)
\left(\frac{w\phi_{\ell}(w^{2},1,q_{a})
\cos(w-\frac{\ell\pi}{2})}{H_{e1}(w)}
\right.\\
&-\left.\frac{(\phi_{\ell}^{\prime}(w^{2},1,q_{a})
+\beta_{(\ell)}\phi_{\ell}(w^{2},1,q_{a})
\sin(w-\frac{\ell\pi}{2})}{H_{e2}(w)}\right),
\end{split}
\end{align}
where $H_{a}(w)=a^{2\ell}H(w/a)$.

Therefore, repeat the process of the proof of the case $a=1$, we conclude that
there exists a potential function $q_{a}\in L^{2}(0,1)$ such that \eqref{e3.7} holds, which implies \eqref{e3.2} holds if $q(x)=q_{a}(x/a)/a^{2}$ on $(0,a)$ and $q=0$ on $(a,1)$.
This completes the proof of Proposition \ref{l3.6}.
\end{proof}

\section{Proofs of Theorems \ref{t1.1}-\ref{c1.1}}\label{s4}

In this section, we will prove Theorems \ref{t1.1}-\ref{c1.1}.

\begin{proof}[The proof of Theorem \ref{t1.1}]
(1)$\Rightarrow$(2). We only need to prove it in the case that $\ell$ is even since the odd case is similar.

Let $E=\{\pm\sqrt{\Lambda}\cup\{\pm a_{k}\}_{k=1}^{\ell}\cup\{0,1\}\}$ with $\pm\sqrt{\Lambda}\cap(\{\pm a_{k}\}_{k=1}^{\ell}\cup\{0,1\})=\emptyset$ and $a_{k}\in (1,\infty)$ for $k=1,\cdots,\ell$.
Recall (cf. \cite[Theorem 1.8]{Sed2003}) that a change of a finite number of values in the exponential system does not influence on the completeness. Thus suppose 
$\{e^{i\lambda t}\mid\lambda\in E\}$  is not complete in $L^{2}(-2a,2a)$, then it follows from Lemma \ref{l3.1} that $E$ is not a uniqueness set for $PW_{2a}$. By Lemma \ref{l3.2}, there exists $f\in PW_{2a}$,  not identically zero, such that $f$ vanishes on $E$ and $f(\mathbb{R})\subseteq\mathbb{R}$.
Also, at least one of the functions $f(z)$ and $f(z)/(z-1)$ is not odd.
So without loss of generality, assume $f$ is not odd.
Define $f_{\textnormal{e}}(z)=f(z)+f(-z)$ and hence $f_{\textnormal{e}}$ is a nonzero even function that vanishes on $E\setminus\{1\}$.
It follows from Proposition \ref{l3.6} that there exists $q\in L^2(0,1)$ such that $q$ and $f_{\textnormal{e}}$ satisfy \eqref{e3.2} and $q=0$ on $(a,1)$.
Similarly, for $\tilde{f}_{\textnormal{e}}=f_{\textnormal{e}}/2$, there exists $\tilde{q} \in L^2(0,1)$ such that $\tilde{q}$ and $\tilde{f}_{\textnormal{e}}$ satisfy \eqref{e3.2} and $\tilde{q}=0$ on $(a,1)$. 

Combining \eqref{e3.2} and the fact that $\tilde{f}_{\textnormal{e}}(\sqrt{\lambda})=f_{\textnormal{e}}(\sqrt{\lambda})=0$ for $\lambda\in\Lambda\cup\{0\}$, we get
\[
\phi_{\ell}(\lambda, a, q)\phi_{\ell}^{\prime}(\lambda, a, \tilde{q})
-\phi_{\ell}(\lambda, a, \tilde{q})\phi_{\ell}^{\prime}(\lambda, a, q)
=0,
~~\lambda\in\Lambda,
\]
which together with \eqref{e3.1} and $q=\tilde{q}=0$ on $(a,1)$ implies 
$m=\tilde{m}$ on $\Lambda$. Therefore, $q=\tilde{q}$ a.e. on $(0,1)$ by the assumption.
Thus, again by \eqref{e3.2}, we have $f_{\textnormal{e}}=f_{\textnormal{e}}/2\equiv 0$, a contradiction, the desired implication follows immediately.

\vspace{0.2cm}

(2)$\Rightarrow$(1). It follows from Proposition \ref{l3.5} that
\begin{align*}
(\phi_{\ell}(z^2, 1, q)\phi_{\ell}^{\prime}(z^2, 1, \tilde{q})
-\phi_{\ell}(z^2, 1, \tilde{q})\phi_{\ell}^{\prime}(z^2, 1, q))
\prod_{n=1}^{2\ell+2}(z-a_{n})=f(z)+C
\end{align*}
for some $f\in PW_{2a}$ and some constant $C$.

Let $x\in\pm \sqrt{\Lambda} \cup\left\{a_{n}\right\}_{n=1}^{2\ell+2}$ with $a_{n}\notin\pm\sqrt{\Lambda}$. Then the condition $m=\tilde{m}$ on $\Lambda$ yields
\begin{align*}
f(x)+C=(\phi_{\ell}(x^2, 1, q)\phi_{\ell}^{\prime}(x^2, 1, \tilde{q})
-\phi_{\ell}(x^2, 1, \tilde{q})\phi_{\ell}^{\prime}(x^2, 1, q))
\prod_{n=1}^{2\ell+2}(x-a_{n})=0.
\end{align*}
Note that $\lim_{t\in\mathbb{R},|t|\rightarrow\infty}f(t)=0$ and hence $C=0$. 
This together with (2) and Lemma \ref{l3.1} implies that $f \equiv 0$, that is $m\equiv\tilde{m}$. 
Therefore, $q=\tilde{q}$ a.e. on $(0,1)$, which completes the proof.
\end{proof}

\begin{proof}[The proof of Theorem \ref{c1.1}]
(1)$\Leftrightarrow$(2). Since $\lambda_{n}\in\sigma(\ell,q,\beta_{n})
\cap\sigma(\ell,\tilde{q},\beta_{n})$, it follows that
\begin{align*}
m(\lambda_{n})=\frac{\phi_{\ell}(\lambda_{n}, 1, q)}
{\phi_{\ell}^{\prime}(\lambda_{n}, 1, q)}=-\frac{1}{\beta_{n}}=\frac{\phi_{\ell}(\lambda_{n}, 1, \tilde{q})}
{\phi_{\ell}^{\prime}(\lambda_{n}, 1, \tilde{q})}=\tilde{m}(\lambda_{n}),
\end{align*}
whence $m=\tilde{m}$ on $\Lambda$. 
Therefore, the desired result follows from Theorem \ref{t1.1}. 

\vspace{0.2cm}

(3)$\Rightarrow$(2).
Recall it follows again from Theorem 1.8 in \cite{Sed2003} that a change of a finite number of values in the exponential system does not influence on the completeness.
Therefore, we select $\ell+1$ distinct numbers $a_{k}\in(0,+\infty)$ such that 
$\pm \sqrt{\Lambda}\cap\{\pm a_{k}\}_{k=1}^{\ell+1}=\emptyset$. Let $E_{1}=\pm\sqrt{\Lambda}\cup\{\pm a_{k}\}_{k=1}^{\ell+1}$.
Suppose on the contrary $\{e^{i\lambda t}\mid \lambda\in E_{1}\}$ is not complete in $L^{2}(-2a, 2a)$, then we can find a nonzero even function $f\in PW_{2a}$ such that $f$ vanishes on $E_{1}$. Indeed, if $f$ is not even, define $f_{1}$ as follows,
\begin{align*}
f_{1}(z)
=\begin{cases}
f(z)/z,                     & \textnormal{if $f$ is odd and $z\in\mathbb{C}\setminus\{0\}$},\\
f(z)+f(-z),               & \textnormal{if $f$ is not odd},
\end{cases}
\end{align*}
where $f_{1}(0)=f'(0)$ if $f$ is odd.
Then $f_{1}$ is the desired function.

From the definition of $PW_{2a}$, it follows that
\[
g(z)=\frac{f(z)z^{2\ell+2}}{\prod_{k=1}^{\ell+1}(z^{2}-a_{k}^{2})} 
\in PW_{2a}
\]
and there also exists a nonzero even function $G \in L^2(-2a, 2a)$ such that
\[
g(z)=\int_{-2a}^{2a}G(t)e^{izt}dt.
\]
Write $g_{1}(z)=2g(\sqrt{z})=g(\sqrt{z})+g(-\sqrt{z})$. It is clear that $g_{1}$ is a nonzero entire function, and $g_{1}(\lambda)=0$ for $\lambda\in\Lambda\cup\{0\}$, and $0$ is a zero of $g_{1}$ of order at least $\ell+1$. Moreover,
\begin{align}\label{e4.2}
g_{1}(z)=2\int_{-2a}^{2a}G(t)\cos(\sqrt{z}t)dt
=8\int_{0}^{a}G(2t)\cos(2\sqrt{z}t)dt,
\end{align}
which implies
\begin{align}\label{e4.3}
0=g_{1}^{(k)}(0)=(-1)^k 8\frac{4^k(k)!}{(2k)!}
\int_{0}^{a}t^{2k}G(2t)dt, 
~~k=0,1,2,\cdots,\ell.
\end{align}
Therefore, combine \eqref{e4.2} and \eqref{e4.3}, we conclude that $G(2t)$ is orthogonal to the system 
$\{\cos(2\sqrt{\lambda}t)\mid\lambda\in\Lambda\}\cup\{t^{2k}\mid k = 0, 1,2,\cdots,\ell\}$. 
Thus $\{\cos(2\sqrt{\lambda}t)\mid\lambda\in\Lambda\}\cup \{t^{2k}\mid k=0,1,2,\cdots,\ell\}$ is not complete in $L^2(0, a)$. This contradiction gives the desired conclusion in (2).

\vspace{0.2cm}

(2)$\Rightarrow$(3). The proof is similar to (3)$\Rightarrow$(2) and hence is omitted. This completes the proof.

\end{proof}

\section{The non-integer case of $\ell$}\label{s5}

We have so far considered the case where $\ell$ is an integer. In this section, we turn to the case where $\ell$ is not an integer.

\begin{lemma}\label{l3.4}
Let $a\in(0,1]$ and $\ell>-1/2$.
Suppose $q$, $\tilde{q}$ belong to $L^{1}(0,1)$ and satisfy $q=\tilde{q}$ on $(a,1)$.
Then
\begin{align*}
f(z)
=z^{\lfloor2\ell\rfloor+1}(\phi_{\ell}(z^2,1,q)\phi_{\ell}^{\prime}(z^2,1,\tilde{q})
-\phi_{\ell}(z^2, 1, \tilde{q})\phi_{\ell}^{\prime}(z^2, 1, q))
\in PW_{2a},
\end{align*}
where $\lfloor2\ell\rfloor$ is the nearest integer to $2\ell$, with the tie-breaker favoring the smaller integer.
\end{lemma}

\begin{proof}
Combining the fact that $q=\tilde{q}$ on $(a,1)$ and \eqref{e3.1}, we have
\begin{align*}
f(z)=&z^{\lfloor2\ell\rfloor+1}
(\phi_{\ell}(z^2,1,q)\phi_{\ell}^{\prime}(z^2,1,\tilde{q})
-\phi_{\ell}(z^2,1,\tilde{q})\phi_{\ell}^{\prime}(z^2,1,q)) \nonumber \\
=&z^{\lfloor2\ell\rfloor+1}\int_{0}^{1}(q-\tilde{q})\phi_{\ell}(z^2, x, q) \phi_{\ell}(z^2, x, \tilde{q})dx \nonumber \\
=&z^{\lfloor2\ell\rfloor+1}\int_{0}^{a} (q-\tilde{q})\phi_{\ell}(z^2, x, q) \phi_{\ell}(z^2, x, \tilde{q})dx \nonumber \\
=&z^{\lfloor2\ell\rfloor+1}(\phi_{\ell}(z^2, a, q)\phi_{\ell}^{\prime}(z^2, a, \tilde{q})
-\phi_{\ell}(z^2, a, \tilde{q})\phi_{\ell}^{\prime}(z^2, a, q)),
\end{align*}
which together with \eqref{e2.1} and \eqref{e2.2} implies that $f$
is an entire function of exponential type at most $2a$  and it is square integrable on $\mathbb{R}$, that is, $f\in PW_{2a}$.
This completes the proof.
\end{proof}

\begin{proof}[The proof of Theorem \ref{t1.2}]
It follows from Lemma \ref{l3.4} that
\begin{align*}
(\phi_{\ell}(z^2, 1, q)\phi_{\ell}^{\prime}(z^2, 1, \tilde{q})
-\phi_{\ell}(z^2, 1, \tilde{q})\phi_{\ell}^{\prime}(z^2, 1, q))
\prod_{k=1}^{\lfloor2\ell\rfloor+1}(z-a_{k})=f(z)
\end{align*}
for some $f\in PW_{2a}$.

Let $x\in\pm \sqrt{\Lambda} \cup\left\{a_{k}\right\}_{k=1}^{\lfloor2\ell\rfloor+1}$ with $a_{n}\notin\pm\sqrt{\Lambda}$. Then the condition $m=\tilde{m}$ on $\Lambda$ yields
\begin{align*}
f(x)=(\phi_{\ell}(x^2, 1, q)\phi_{\ell}^{\prime}(x^2, 1, \tilde{q})
-\phi_{\ell}(x^2, 1, \tilde{q})\phi_{\ell}^{\prime}(x^2, 1, q))
\prod_{k=1}^{\lfloor2\ell\rfloor+1}(x-a_{k})=0,
\end{align*}
which together with Lemma \ref{l3.1} implies $f \equiv 0$, that is, $m\equiv\tilde{m}$. 
Therefore, $q=\tilde{q}$ a.e. on $(0,1)$.
This completes the proof.
\end{proof}

Applying Theorem \ref{t1.2}, we next establish the Borg-type \cite{Borg1946} and Hochstadt-Lieberman-type \cite{HB1978} results for equation \eqref{e1.1}.

Lemma \ref{l4.1} is a simple application of the Levinson's completeness test (see \cite[Theorem~3, pp.118]{Young1980}), so the proof is omitted.

\begin{lemma}\label{l4.1}
Suppose $\{\varepsilon_{n}\}_{n=0}^{\infty}$ is a sequence of positive numbers such that $\sum\varepsilon_{n}/(n+1)<\infty$. Then the following hold.

\begin{enumerate}

 \item [\textnormal{(1)}] \cite[Problem 4, pp.122]{Young1980} The system $\{e^{i\gamma_0 t}\}\cup\{e^{\pm i\gamma_{n} t}\mid n\in\mathbb{N}\}$ is complete in $L^2(-1,1)$, whenever
   \begin{align*}
   |\gamma_{n}|\leq n\pi+\frac{\pi}{4}+\varepsilon_{n}, ~~n\in \mathbb{N}_{0}
   \end{align*}
and $|\gamma_0|<|\gamma_1|<|\gamma_2|<\cdots$.

\item [\textnormal{(2)}]  The system $\{e^{\pm i\gamma_{n} t}\mid n\in\mathbb{N}\}$ is complete in $L^{2}(-1,1)$, whenever
\begin{align*}
|\gamma_{n}|\leq n\pi-\frac{\pi}{4}+\varepsilon_{n}, ~~n\in\mathbb{N}
\end{align*}
and $0<|\gamma_1|<|\gamma_2|<\cdots$.

\end{enumerate}
\end{lemma}

We first obtain the Borg-type theorem (see \cite[Theorem~1.3]{Carl1997} and \cite[Theorem~2.8]{KST2010}) using Theorem \ref{t1.2}. For convenience, we assume the Dirichlet and Neumann boundary conditions at $x=1$, respectively. 

\begin{corollary}\label{exa1}
Suppose $\ell>-1/2$, $q, \tilde{q} \in L^p(0,1)$ for $p \in[1, \infty]$, $\sigma(\ell, q, \infty)=\sigma(\ell, \tilde{q}, \infty)$ and $\sigma(\ell, q, 0)=\sigma(\ell, \tilde{q}, 0)$. Then $q=\tilde{q}$ a.e. on $(0,1)$.
\end{corollary}

\begin{proof}
Let $\lambda_{\ell, \infty, n}=\lambda_{\ell, \infty, n}(q)$ and $\lambda_{\ell, 0, n}=\lambda_{\ell, 0, n}(q)$. Then $\lambda_{\ell, \infty, n}(\tilde{q})=\lambda_{\ell, \infty, n}$ and $\lambda_{\ell, 0, n}(\tilde{q})=\lambda_{\ell, 0, n}$.

Recall that the operator $L(\ell, q)$ corresponding to \eqref{e1.2} is bounded below and has a simple discrete spectrum $\sigma(\ell, q, \beta)$. 
Since
\begin{align*}
-u^{\prime \prime}(x)+\frac{\ell(\ell+1)}{x^{2}}u(x)+(q(x)+c)u(x)
=(\lambda+c)u(x),
\end{align*}
without loss of generality, we assume $\lambda_{\ell, \infty, n}>0$ and $\lambda_{\ell, 0, n}>0$ for all $n\in\mathbb{N}$. It follows from $\sigma(\ell, q, \infty)=\sigma(\ell, \tilde{q}, \infty)$ and $\sigma(\ell, q, 0)=\sigma(\ell, \tilde{q}, 0)$ that
\begin{align}\label{e4.1}
m=\tilde{m}~~
\textnormal{on} ~~\{\lambda_{\ell, \infty, n}\mid n\in\mathbb{N}\} \cup\{\lambda_{\ell, 0, n}\mid n\in\mathbb{N}\}.
\end{align}
In the following, we will prove $q=\tilde{q}$ a.e. on $(0,1)$ and we divide it into four cases.

\vspace{0.1cm}

Case 1. $N \leq \ell \leq N+1/4$ or $N+3/4<\ell \leq N+1$ for some $N \in \mathbb{N}_{0}$. In this case $\lfloor2\ell\rfloor=2\lfloor\ell\rfloor$. 

For every $k=0, 1, \cdots, \lfloor\ell\rfloor$, choose $\mu_{k}\in\mathbb{R}$ such that
\[
\pm \mu_k \notin \{\sqrt{\lambda_{\ell, \infty, n}} \mid n\in\mathbb{N}\} \cup \{\sqrt{\lambda_{\ell, 0, n}} \mid n\in\mathbb{N}\}.
\]
Rewrite the set
\begin{align*}
\{2 \mu_k \mid k=0,1, \cdots,\lfloor\ell\rfloor\} 
\cup\{2 \sqrt{\lambda_{\ell, \infty, n}} \mid n\in\mathbb{N}\} 
\cup\{2 \sqrt{\lambda_{\ell, 0, n}} \mid n\in\mathbb{N}\}
\end{align*}
as $\{\gamma_n \mid n\in\mathbb{N}_{0}\}$ so that $|\gamma_0|<|\gamma_1|<|\gamma_2|<\cdots$.
The asymptotic estimates \eqref{ew1} and \eqref{ew2} yield
\begin{align*}
|\gamma_n| \leq n\pi+\frac{\pi}{4}+O\left(\frac{1}{n}\right),
\end{align*}
which together with Lemma \ref{l4.1} (1) implies that $\{e^{i\gamma_{0}t}\} \cup\{e^{ \pm i\gamma_{n}t}\mid n\in\mathbb{N}\}$ 
is complete in $L^{2}(-1,1)$. Then it follows from \eqref{e4.1} and Theorem \ref{t1.2} that $q=\tilde{q}$ a.e. on $(0,1)$.

\vspace{0.1cm}

Case 2. $N+1/4<\ell\leq N+1/2$ for some $N \in \mathbb{N}_{0}$. In this case $\lfloor2\ell\rfloor=2\lfloor\ell\rfloor+1$.

For every  $k=1,\cdots,\lfloor\ell\rfloor+1$, choose $\mu_{k}\in\mathbb{R}$ such that 
\[\pm \mu_{k}\notin\{\sqrt{\lambda_{\ell,\infty,n}} \mid n\in\mathbb{N}\} \cup\{\sqrt{\lambda_{\ell, 0,n}} \mid n\in\mathbb{N}\}
\]
and $|\mu_{k}|>0$ for $k=1,\cdots,\lfloor\ell\rfloor+1$. Rewrite the set 
\begin{align*}
\{2\mu_k\mid k=1,\cdots,\lfloor\ell\rfloor+1\} 
\cup\{2\sqrt{\lambda_{\ell,\infty,n}}\mid n\in\mathbb{N}\}
\cup\{2\sqrt{\lambda_{\ell,0,n}}\mid n\in\mathbb{N}\}
\end{align*}
as $\{\gamma_n \mid n\in\mathbb{N}\}$ so that $0<|\gamma_{1}|<|\gamma_{2}|<\cdots$.
The asymptotic estimates \eqref{ew1} and \eqref{ew2} yield
\begin{align*}
|\gamma_n|<n \pi-\frac{\pi}{4}+O\left(\frac{1}{n}\right),
\end{align*}
which together with Lemma \ref{l4.1} (2) implies that
$\{e^{ \pm i\gamma_{n}t} \mid n\in\mathbb{N}\}$ is complete in $L^{2}(-1,1)$.
Then \eqref{e4.1} and Theorem \ref{t1.2} gives $q=\tilde{q}$ a.e. on $(0,1)$.

\vspace{0.1cm}

Case 3. $N+1/2<\ell\leq N+3/4$ for some $N\in\mathbb{N}_{0}$ or $-1/2<\ell\leq-1/4$. The proof is similar to Case 2.

\vspace{0.1cm}

Case 4. $-1 / 4<\ell \leq 0$. This case is similar to Case 1.

This completes the proof.
\end{proof}

In the following, By using Theorem \ref{t1.2}, we study the half-inverse problem, also called the Hochstadt-Lieberman-type problem.
Recall that Koyunbakan and Panakhov \cite{HE2005} showed that if $q(x)$ is known on the half-interval $(1/2,1)$, then it can be recovered on $(0,1/2)$ from a single spectrum $\sigma(\ell, q, \beta)$. However, Corollary \ref{exa2} shows that one can remove $n_{\ell}(\beta)$ elements from $\sigma(\ell, q, \beta)$, thereby strengthening Koyunbakan and Panakhov's result. In fact, for $\ell\geq0$ and $\beta=\infty$, a similar result in Corollary \ref{exa2} was given by Liu, Shi, and Yan \cite{YGJ2019}. For $\ell \in \mathbb{N}$ and $q, \tilde{q} \in L^p(0,1)$ with $p \in (1, \infty)$, our result coincides with \cite[Theorem~4.2]{XYB2023}.

\begin{corollary}\label{exa2}
Suppose $\ell\geq0$, $q, \tilde{q} \in L^p(0,1)$ for $p \in[1, \infty]$ and $q=\tilde{q}$ a.e. on $(1/2, 1)$. Suppose the two sets $\sigma(\ell, q, \beta)$ and $\sigma(\ell, \tilde{q}, \beta)$ are equal if we remove $n_{\ell}(\beta)$ elements from each of them, where
\begin{align*}
n_{\ell}(\beta):=
\begin{cases} 
\left[\frac{\ell}{2}\right], & \beta = \infty, \\ 
\left[\frac{\ell+1}{2}\right], & \beta \in \mathbb{R}.
\end{cases}
\end{align*}
and the value $[ x]$ for $x \in \mathbb{R}$ equals the maximum integer not exceeding $x$. 
Then $q=\tilde{q}$ a.e. on $(0,1)$.
\end{corollary}

\begin{proof}
The proof follows an argument similar to that of Corollary \ref{exa1}. By applying the asymptotic estimates \eqref{ew1} and \eqref{ew2}, Lemma \ref{l4.1} and Theorem \ref{t1.2} (with $a=1/2$), we deduce that $n_{\ell}(\beta)$ eigenvalues can be removed. The remaining eigenvalues are sufficient to uniquely determine the potential $q(x)$ on the interval $(0, 1/2)$. This completes the proof.
\end{proof}

\appendix
\renewcommand{\thetheorem}{\Alph{section}.\arabic{theorem}}
\section{The product formulas of the solutions}\label{s6}

In the appendix, the product formulas of $\phi_{\ell}$ and $\phi_{\ell}^{\prime}+\beta\phi_{\ell}$ are provided. To this end, we need the following lemma, which generalizes \cite[Lemma 2, pp.167]{PT1987}, originally established for $\ell=0$, to any $\ell\in\mathbb{N}_{0}$. Its proof is similar to that for the case $\ell=0$ and hence is omitted.


\begin{lemma}\label{l2.1}
Suppose $\ell\in\mathbb{N}_{0}$ and $\{\lambda_{n}\}_{n\geq 1}$ is a sequence of complex numbers such that
$\lambda_{n}= (n+\ell/2)^{2}\pi^{2}+O(1)$.
Then the infinite product
\begin{align*}
\prod_{n=1}^{\infty} \frac{\lambda_{n}-\lambda}{\left(n+\frac{\ell}{2}\right)^{2} \pi^{2}}
\end{align*}
is an entire function of $\lambda$, whose roots are precisely $\lambda_{n}$, $n\geq 1$. Moreover,
\begin{align}\label{e2.4}
\prod_{n=1}^{\infty}
\frac{\lambda_{n}-\lambda}{\left(n+\frac{\ell}{2}\right)^{2}\pi^{2}}
=C_{\ell}P_{\ell}(\lambda)\frac{\sin\left(\sqrt{\lambda}-\frac{\ell\pi}{2}\right)}
{\lambda^{\frac{\ell+1}{2}}}\left(1+O\left(\frac{\log j}{j}\right)\right) 
\end{align}
uniformly on the circles 
$|\lambda|=(j+(\ell+1)/2)^{2}\pi^{2}$ for each $j\in\mathbb{N}$, where
\begin{align}\label{e2.3}
C_{\ell}=
\begin{cases}
\prod_{n=1}^{\frac{\ell}{2}}n^{2}\pi^{2}, &\textnormal{if}~\ell~\textnormal{is even}, \\
\prod_{n=1}^{\frac{\ell+1}{2}}\left(n-\frac{1}{2}\right)^{2}\pi^{2}, &\textnormal{if}~\ell~\textnormal{is odd},
\end{cases}
\end{align}
and
\begin{align}\label{e2.31}
P_{\ell}(\lambda)=
\begin{cases}
\prod_{n=1}^{\frac{\ell}{2}} 
\frac{\lambda}{\lambda-n^{2}\pi^{2}}, &\textnormal{if}~\ell~\textnormal{is even}, \\
\prod_{n=1}^{\frac{\ell+1}{2}} 
\frac{\lambda}
{\lambda-\left(n-\frac{1}{2}\right)^{2}\pi^{2}}, &\textnormal{if}~\ell~\textnormal{is odd}.
\end{cases}
\end{align}
\end{lemma}

Now we get the product formulas of $\phi_{\ell}$ and $\phi_{\ell}^{\prime}+\beta\phi_{\ell}$.

\begin{proposition}\label{p2.2}
Suppose $q\in L^{1}(0,1)$, $\ell\in\mathbb{N}_{0}$ and $\beta\in\mathbb{R}$. Then 
\begin{align}
\phi_{\ell}(\lambda, 1, q) 
=&\frac{1}{C_{\ell}}\prod_{n=1}^{\infty}
\frac{\lambda_{\ell, \infty, n}(q)-\lambda}
{(n+\frac{\ell}{2})^{2} \pi^{2}},\label{e2.5}\\
\phi_{\ell}^{\prime}(\lambda, 1, q)+\beta\phi_{\ell}(\lambda, 1, q)
=&\frac{1}{C_{\ell-1}}\prod_{n=1}^{\infty}
\frac{\lambda_{\ell, \beta, n}(q)-\lambda}{(n+\frac{\ell-1}{2})^{2} \pi^{2}}, \label{e2.6}
\end{align}
where $C_{\ell}$ is defined by \eqref{e2.3}.
\end{proposition}

\begin{proof}
It follows from \eqref{e2.1} that $\phi_{\ell}(\lambda, 1, q)$ is an exponential type entire function of order at most $1/2$, and then by Hadamard's theorem (see \cite[Section~1.10]{Lev1964}), there exists a constant $c_{1}(q)$ depending on $q$  such that
\begin{align*}
\phi_{\ell}(\lambda, 1, q) 
=c_{1}(q)\prod_{n=1}^{\infty}
\frac{\lambda_{\ell, \infty, n}(q)-\lambda}
{\lambda_{\ell, \infty, n}(q)}
=c_{1}(q)\prod_{n=1}^{\infty}
\frac{(n+\frac{\ell}{2})^{2} \pi^{2}}
{\lambda_{\ell, \infty, n}(q)}
\prod_{n=1}^{\infty}
\frac{\lambda_{\ell, \infty, n}(q)-\lambda}
{(n+\frac{\ell}{2})^{2}\pi^{2}}.
\end{align*}
By \eqref{ew1}, we get $\lambda_{\ell, \infty, n}(q)=(n+\ell/2)^{2}\pi^{2}+O(1)$. Then 
\begin{align*}
\prod_{n=1}^{\infty}
\frac{(n+\frac{\ell}{2})^{2} \pi^{2}}
{\lambda_{\ell, \infty, n}(q)},~~~~
\prod_{n=1}^{\infty}
\frac{\lambda_{\ell, \infty, n}(q)-\lambda}
{(n+\frac{\ell}{2})^{2}\pi^{2}}
\end{align*}
are two entire functions and hence we deduce
\begin{align*}
\phi_{\ell}(\lambda, 1, q) 
=c_{1}(q)\prod_{n=1}^{\infty}
\frac{(n+\frac{\ell}{2})^{2} \pi^{2}}
{\lambda_{\ell, \infty, n}(q)}
\prod_{n=1}^{\infty}
\frac{\lambda_{\ell, \infty, n}(q)-\lambda}
{(n+\frac{\ell}{2})^{2}\pi^{2}}
=c_{2}(q)\prod_{n=1}^{\infty}
\frac{\lambda_{\ell, \infty, n}(q)-\lambda}
{(n+\frac{\ell}{2})^{2}\pi^{2}}.
\end{align*}

If  $|\lambda|=(j+(\ell+1)/2)^{2}\pi^{2}$,
then for any integer $k$, it holds 
\begin{align*}
\left|\sqrt{\lambda}-\frac{\ell\pi}{2}-k\pi\right|
\geq&\left||\sqrt{\lambda}|-\left|\frac{\ell\pi}{2}+k\pi\right|\right|\\
\geq&\min\left\{\left|\left(j-k+\frac{1}{2}\right)\pi\right|,
\left|\left(j+k+\ell+\frac{1}{2}\right)\pi\right|\right\}\\
\geq&\frac{\pi}{2}.
\end{align*}
Hence by \cite[Lemma~1, pp.27]{PT1987}, we have
\begin{align*}
e^{\left|\textnormal{Im}(\sqrt{\lambda}-\frac{\ell \pi}{2})\right|}
<4\left|\sin\left(\sqrt{\lambda}-\frac{\ell \pi}{2}\right)\right|.
\end{align*}
Thus for $|\lambda|=(j+(\ell+1)/2)^{2}\pi^{2}$, \eqref{e2.1} can be rewritten as
\begin{align*}
\phi_{\ell}(\lambda, 1, q) 
=&\lambda^{-\frac{\ell+1}{2}} \sin \left(\sqrt{\lambda}-\frac{\ell \pi}{2}\right)
+O\left(|\lambda|^{-\frac{\ell+2}{2}} e^{|\textnormal{Im}\left(\sqrt{\lambda}-\frac{\ell \pi}{2}\right)|}\right)\\
=&\lambda^{-\frac{\ell+1}{2}} \sin \left(\sqrt{\lambda}-\frac{\ell \pi}{2}\right)
+O\left(|\lambda|^{-\frac{\ell+2}{2}} 
\left| \sin \left(\sqrt{\lambda}-\frac{\ell \pi}{2}\right) \right| \right)\\
=&\lambda^{-\frac{\ell+1}{2}}
\sin \left(\sqrt{\lambda}-\frac{\ell \pi}{2}\right)
\left(1+O\left(\frac{1}{j}\right)\right).
\end{align*}
Also, for $|\lambda|=(j+(\ell+1)/2)^{2}\pi^{2}$, \eqref{e2.4} yields 
\begin{align*}
\phi_{\ell}(\lambda, 1, q)
=c_{2}(q)C_{\ell}P_{\ell}(\lambda)\lambda^{-\frac{\ell+1}{2}}
\sin\left(\sqrt{\lambda}-\frac{\ell\pi}{2}\right)
\left(1+O\left(\frac{\log j}{j}\right)\right). 
\end{align*}

Consequently, for $|\lambda|=(j+(\ell+1)/2)^{2}\pi^{2}$, it follows that
\begin{align*}
c_{2}(q)C_{\ell}P_{\ell}(\lambda)\left(1+O\left(\frac{\log j}{j}\right)\right)
\left(1+O\left(\frac{1}{j}\right)\right)^{-1}=1.
\end{align*}
Recall that $P_{\ell}(\lambda)\rightarrow 1$ as $j\rightarrow\infty$ and hence 
$c_{2}(q)=1/C_{\ell}$.  So the equality \eqref{e2.5} holds.

To prove \eqref{e2.6}, it follows from \eqref{e2.1} and \eqref{e2.2} that
\begin{align*}
&\phi_{\ell}^{\prime}(\lambda, 1, q)+\beta\phi_{\ell}(\lambda, 1, q)\\
=&\lambda^{-\frac{\ell}{2}}\left(\cos\left(\sqrt{\lambda}-\frac{\ell \pi}{2}\right)+\beta\lambda^{-\frac{1}{2}}\sin\left(\sqrt{\lambda}-\frac{\ell \pi}{2}\right)+O\left(|\lambda|^{-\frac{1}{2}} e^{|\operatorname{Im}(\sqrt{\lambda})|}\right)\right) \\
=&\lambda^{-\frac{\ell}{2}}\left(\cos \left(\sqrt{\lambda}-\frac{\ell \pi}{2}\right)+O\left(|\lambda|^{-\frac{1}{2}} e^{|\operatorname{Im}(\sqrt{\lambda})|}\right)\right) \\
=&\lambda^{-\frac{(\ell-1)+1}{2}}\left(\sin \left(\sqrt{\lambda}-\frac{(\ell-1) \pi}{2}\right)+O\left(|\lambda|^{-\frac{1}{2}} e^{|\operatorname{Im}(\sqrt{\lambda})|}\right)\right).
\end{align*}
Then repeat the process of the proof of \eqref{e2.5}, we conclude that the equality \eqref{e2.6} also holds.
This completes the proof.
\end{proof}

\textbf{Acknowledgements.} The research is partially supported by National Natural Science Foundation of China (No. 12571136, No. 12171075 and No. 12001089).

\end{document}